\documentclass[a4paper,12pt]{amsart}
\pdfoutput=1
\usepackage{setspace}
\usepackage[left=18mm,head=30mm,bottom=30mm,foot=20mm,right=18mm]{geometry}
\usepackage{amsmath,amsthm,amsfonts,amssymb,mathtools,mathrsfs,url,bm,enumitem,stackengine}
\usepackage[varvw]{newtxmath}
\usepackage{newtxtext,newtxmath}
\usepackage{hyperref}
\hypersetup{
  colorlinks=true,
  linkcolor=blue,
  citecolor=red,
}
\mathtoolsset{showonlyrefs=true}
\theoremstyle{definition}
  \newtheorem{dfn}{Definition}[section]
  \newtheorem{rem}{Remark}[section]
\theoremstyle{plain}
  \newtheorem{thm}{Theorem}[section]
  \newtheorem{cor}{Corollary}[section]
  \newtheorem{prop}{Proposition}[section]
  \newtheorem{lem}{Lemma}[section]
  
\renewcommand{\Re}{\operatorname{Re}}

\newcommand\restr[2]{{
  \left.\kern-\nulldelimiterspace 
  #1 
  \vphantom{\big|} 
  \right|_{#2} 
}}
\DeclareMathOperator{\sgn}{sgn}
\begin{document}

\title[Long-time behavior of several point particles]{Long-time behavior of several point particles in a 1D viscous compressible fluid}
\author{Kai Koike}
\address{Department of Mathematics, Tokyo Institute of Technology, Tokyo 152-8551, Japan}
\email{koike.k@math.titech.ac.jp}
\date{\today}

\begin{abstract}
  We study the long-time behavior of \textit{several} point particles in a 1D viscous compressible fluid. It is shown that the velocities of the point particles all obey the power law $t^{-3/2}$. This result extends author's previous works on the long-time behavior of a \textit{single} point particle. New difficulties arise in the derivation of pointwise estimates of Green's functions due to infinite reflections of waves in-between the point particles. In particular, the differential equation technique used in previous works alone does not suffice. We overcome this by carefully analyzing the structure of Green's functions in the Laplace variable, especially their asymptotic and analyticity properties.
\end{abstract}

\maketitle

\tableofcontents

\section{Introduction}
\subsection{Long-time behavior of a solid in a fluid}
Fluid--structure interaction problems deal with phenomena caused by interaction of moving or deforming solids with fluid flows. Mathematically, this requires us to simultaneously solve PDEs for the fluids and ODEs for the solids. This brings in new aspects to mathematical analysis of fluid dynamical equations and has attracted attention of many mathematicians.

Amongst various interesting aspects of fluid--structure interaction problems, we focus here on the problem of long-time behavior of moving solids. There are several works on this problem, and we shall briefly review these results. In~\cite{Liu78}, Liu analyzed the motion of a point particle\footnote{Some authors prefer to call it \textit{a piston}. This is just a matter of taste. We chose the terminology \textit{a point particle} to emphasise that the problem is considered in a one-dimensional setting.}~in a 1D inviscid compressible fluid: he proved that, starting from a small perturbation of a constant state, the velocity $V(t)$ of the point particle decays at least as $t^{-3/2}$. Afterwards, another important result was obtained in the work by Vázquez and Zuazua~\cite{VZ03}: they considered the motion of a point particle in a 1D viscous Burgers fluid and showed that the velocity $V(t)$ of the point particle obeys a power law $t^{-1/2}$ (no restriction on the size of the initial data required). In multi-dimension, Ervedoza, Hillairet, and Lacave considered the motion of a disk in a 2D viscous incompressible fluid and showed that the velocity $V(t)\in \mathbb{R}^2$ of the disk decays at least as $t^{-(2-\epsilon)/2}$ (here $\epsilon$ is a positive number that can be taken arbitrary small for sufficiently regular and small initial data)~\cite{EHL14}. Moreover, this result was recently extended to the 3D case by Ervedoza, Maity, and Tucsnak~\cite{EMT20}: they showed the decay estimate $|V(t)|\lesssim t^{-(3-\epsilon)/2}$.

All of the works mentioned above consider the motion of a solid in \textit{an unbounded fluid domain}. Of course, there are works dealing with the motion of a solid in \textit{a bounded fluid domain}. We refer, for example, to~\cite{FMNT18,Lequeurre20,MTT17,Shelukhin78} and the references therein. In this case, the velocity $V(t)$ of the solid usually decays exponentially fast, and the interest of the research mainly lies in constructing and analyzing solutions without restrictions on the size of initial data.

One of the missing peaces in the works mentioned above was the corresponding problem for 1D viscous compressible fluids. For this problem, we recently showed in~\cite{Koike21} that the velocity $V(t)$ of a point particle moving in a 1D viscous compressible fluid decays at least as $t^{-3/2}$ for small initial data. Moreover, in~\cite{Koike20-p}, we revealed a simple necessary and sufficient condition on the initial data guaranteeing the optimality of the decay rate $-3/2$. These results were obtained by applying and refining the method of pointwise estimates of Green's function developed and used, for example, in~\cite{Deng16,DW18,LY11,LY12,LZ97,Zeng94}.

\subsection{Long-time behavior of several solids in a fluid}
Given these basic understanding of long-time behavior of a \textit{single} solid in a fluid, a natural question comes up: what happens when there are \textit{several} solids? This question was addressed, for example, in~\cite{VZ06} where the authors studied the long-time behavior of several point particles in a 1D viscous Burgers fluid. One of the interesting results of their paper is that collisions between point particles do not occur in finite time; interestingly, they also showed that the distances between the particles may asymptotically converge to zero. We also refer to~\cite{Hillairet05} for a related result. These works then prompted investigations on the possibility of collisions between solids in multi-dimensions: see, e.g.,~\cite{FN11,Hesla04,Hillairet07,HT09,HM08,MR15,NP10,Sabbagh19} for such results.

Now what happens for the motion of several point particles in the 1D viscous compressible fluid considered in~\cite{Koike20-p,Koike21}? In this paper, we answer to this question. As it turns out, collisions between point particles do not occur; neither do they collide in finite or infinite time. This is because collisions imply unbounded growth of the fluid pressure at the contact point (see~Section~\ref{subsec:NoCollisions} below). This part is rather easy, and the difficult part is the analysis of the decay property of the velocities of the point particles. The main technical difficulty is the derivation of sharp enough pointwise estimates of Green's functions. The presence of several point particles makes the expression of Green's functions in the Laplace variable quite complex, and the differential equation technique used in previous works (e.g.~\cite{Deng16,DW18,Koike21}) alone does not suffice. We overcome this difficulty by carefully analyzing the structure of Green's functions in the Laplace variable, especially their asymptotic and analyticity properties (Section~\ref{sec:PWE_GreensFunctions}). As a consequence, we prove that the velocities of the point particles all decay as $t^{-3/2}$.

In the rest of this section, we give the formulation of the problem. The main theorems are presented in Section~\ref{sec:main_thm}. The proof is given in Section~\ref{sec:proofs}.

\subsection{Motion of two point particles in a 1D viscous compressible fluid: Formulation}
Let us explain the equations we consider in this paper. Although it is possible to treat three or more point particles, let us restrict ourselves to two point particles for simplicity.

Consider a one-dimensional flow in the real line $\mathbb{R}$, and let $X$ be a Cartesian coordinate on $\mathbb{R}$. Denote by $\rho=\rho(X,t)$ and $U=U(X,t)$ the density and the velocity of the fluid. We assume that the fluid is viscous with a constant viscosity coefficient $\nu>0$ and that the fluid is barotropic, that is, the pressure $P$ is a function only of the density $\rho$: $P=P(\rho)$. For the point masses, denote by $X=h_0(t)$ and $X=h_1(t)$ the locations of point particles number 0 and 1; we assume that $h_0(t)<h_1(t)$. Their velocities are denoted by $V_0(t)=h_{0}'(t)$ and $V_1(t)=h_{1}'(t)$. For simplicity, we assume that the masses of the point particles are both unity.

With the assumptions and notations above, the fluid--solids system is governed by the following equations (cf.~\cite[Section~1.1]{Koike20-p}):
\begin{equation}
  \label{eq:FundEqEuler}
  \begin{dcases}
    \rho_t+(\rho U)_X=0,                                                          & X\in \mathbb{R}\backslash \{ h_0(t),h_1(t) \},\, t>0, \\
    (\rho U)_t+(\rho U^2)_X+P(\rho)_X=\nu U_{XX},                    & X\in \mathbb{R}\backslash \{ h_0(t),h_1(t) \},\, t>0, \\
    U(h_0(t)_{\pm},t)=V_0(t),\, U(h_1(t)_{\pm},t)=V_1(t),                         & t>0, \\
    V_{0}'(t)=\llbracket -P(\rho)+\nu U_X \rrbracket(h_0(t),t),                   & t>0, \\
    V_{1}'(t)=\llbracket -P(\rho)+\nu U_X \rrbracket(h_1(t),t),                   & t>0, \\
    h_0(0)=h_{0}^{0},\, V_0(0)=V_{0}^{0},\, h_1(0)=h_{1}^{0},\, V_1(0)=V_{1}^{0}, & \\
    \rho(X,0)=\rho_0(X),\, U(X,0)=U_0(X),                                         & X\in \mathbb{R}\backslash \{ h_{0}^{0},h_{1}^{0} \}.
  \end{dcases}
\end{equation}
Here, $f(X_+,t)$ and $f(X_-,t)$ denote $\lim_{Y\searrow X}f(Y,t)$ and $\lim_{Y\nearrow X}f(Y,t)$, respectively; note also that when we write $f(X_{\pm},t)=g(t)$, this means that $f(X_+,t)=f(X_-,t)=g(t)$. The double brackets denote the jump of a function inside them: $\llbracket f \rrbracket(X,t)=f(X_+,t)-f(X_-,t)$.  The first two equations are the 1D barotropic compressible Navier--Stoke equations, and the equations in the third line are boundary conditions for them. The equations in the fourth and the fifth lines are Newton's equations of motion for the point particles. The rest are initial conditions.

Note that the equations above are posed in a time-dependent domain $\mathbb{R}\backslash \{ h_0(t),h_1(t) \}$. To cast the domain into a time-independent one, we introduce \textit{the Lagrangian mass coordinate}. We assume for simplicity that $\int_{h_{0}^{0}}^{h_{1}^{0}}\rho_0(X)\, dX=1$. Then by~\eqref{eq:FundEqEuler}, we have $\int_{h_0(t)}^{h_1(t)}\rho(X,t)\, dX=1$ for $t\geq 0$. Now, fix $x\in \mathbb{R}_* \coloneqq \mathbb{R}\backslash \{ 0,1 \}$ and $t\geq 0$, and let $X=X(x,t)$ be the solution to
\begin{equation}
  \label{eq:LagrangianMassCoordinate}
  x=\int_{h_0(t)}^{X(x,t)}\rho(X',t)\, dX'.
\end{equation}
Let us assume that $\rho(X,t)\geq \rho_0$ for some $\rho_0>0$ (we only consider such solutions in this paper). Then~\eqref{eq:LagrangianMassCoordinate} is uniquely solvable and determines a one-to-one map
\begin{equation}
  \mathbb{R}_* \ni x \mapsto X(x,t)\in \mathbb{R}\backslash \{ h_0(t),h_1(t) \}.
\end{equation}
This new coordinate $x$ is the Lagrangian mass coordinate. Now using this, we define
\begin{equation}
  v(x,t)=\frac{1}{\rho(X(x,t),t)}, \quad u(x,t)=U(X(x,t),t), \quad p(v)=P\left( \frac{1}{v} \right).
\end{equation}
The quantity $v$ is called \textit{the specific volume} of the fluid. Note that by using~\eqref{eq:FundEqEuler}, it follows that 
\begin{equation}
  \frac{\partial X(x,t)}{\partial x}=v, \quad \frac{\partial X(x,t)}{\partial t}=u.
\end{equation}
Then we can show that, in terms of these new variables,~\eqref{eq:FundEqEuler} is equivalent to
\begin{equation}
  \label{eq:FundEq}
  \begin{dcases}
    v_t-u_x=0,                                                                & x\in \mathbb{R}_*,\, t>0, \\
    u_t+p(v)_x=\nu \left( \frac{u_x}{v} \right)_x,                            & x\in \mathbb{R}_*,\, t>0, \\
    u(0_{\pm},t)=V_0(t),\, u(1_{\pm},t)=V_1(t),                               & t>0, \\
    V_{0}'(t)=\llbracket -p(v)+\nu u_x/v \rrbracket(0,t),                     & t>0, \\
    V_{1}'(t)=\llbracket -p(v)+\nu u_x/v \rrbracket(1,t),                     & t>0, \\
    V_0(0)=V_{0}^{0},\, V_{1}(0)=V_{1}^{0};\, v(x,0)=v_0(x),\, u(x,0)=u_0(x), & x\in \mathbb{R}_*.
  \end{dcases}
\end{equation}
Here,
\begin{equation}
  v_0(x)=\frac{1}{\rho_0(X(x,0))}, \quad u_0(x)=U_0(X(x,0)).
\end{equation}
We note that~\eqref{eq:FundEq} does not contain $h_0(t)$ and $h_1(t)$, but we can recover them by $h_0(t)=h_{0}^{0}+\int_{0}^{t}V_0(s)\, ds$ and $h_1(t)=h_{1}^{0}+\int_{0}^{t}V_1(s)\, ds$.

\section{Main theorems}
\label{sec:main_thm}
The main theorems of this paper concern pointwise estimates of solutions to~\eqref{eq:FundEq}, from which results on the long-time behavior of point particles are derived as corollaries. These results are extensions of the results in~\cite{Koike20-p,Koike21} to several point particles.

\subsection{Preliminaries}
To state the main theorems, we start with some preliminaries. First, we study the structure of the linearized equations of the first two equations in~\eqref{eq:FundEq} around the constant state $(v,u)=(v^*,0)$. Here, the reference specific volume $v=v_*$ can in fact be any positive number but we set $v^*=1$ for simplicity. Then the linearized equations can be written as
\begin{equation}
  \label{eq:FundEqLin}
  \bm{u}_t+A\bm{u}_x=B\bm{u}_{xx}+
  \begin{pmatrix}
    0 \\
    N_x
  \end{pmatrix},
\end{equation}
where
\begin{equation}
  \label{eq:vectors_matrices}
  \bm{u}=
  \begin{pmatrix}
    v-1 \\
    u
  \end{pmatrix},
  \quad A=
  \begin{pmatrix}
    0 & -1 \\
    -c^2 & 0
  \end{pmatrix},
  \quad B=
  \begin{pmatrix}
    0 & 0 \\
    0 & \nu
  \end{pmatrix},
  \quad N=-p(v)+p(1)-c^2(v-1)-\nu \frac{v-1}{v}u_x.
\end{equation}
Here, $c>0$ is the speed of sound for the state $(v,u)=(1,0)$ defined by $c^2=-p'(1)$; for $c$ to be well-defined, we assume that $p'(1)<0$. The matrix $A$ has two eigenvalues $\lambda_1=c$ and $\lambda_2=-c$, and as right and left eigenvectors of $A$ corresponding to $\lambda_i$, we can take $r_i$ and $l_i$, respectively, as follows:
\begin{equation}
  r_1=\frac{2c}{p''(1)}
  \begin{pmatrix}
    -1 \\
    c
  \end{pmatrix},
  \quad r_2=\frac{2c}{p''(1)}
  \begin{pmatrix}
    1 \\
    c
  \end{pmatrix}
\end{equation}
and
\begin{equation}
  l_1=\frac{p''(1)}{4c}
  \begin{pmatrix}
    -1 & 1/c
  \end{pmatrix},
  \quad l_2=\frac{p''(1)}{4c}
  \begin{pmatrix}
    1 & 1/c
  \end{pmatrix}.
\end{equation}
Here and in what follows, we assume that $p''(1)\neq 0$.

We next decompose $\bm{u}={}^{t}(v-1,u)$ with respect to the eigenbasis $(r_1,r_2)$:
\begin{equation}
  \bm{u}=u_1 r_1+u_2 r_2.
\end{equation}
Taking into account the relation
\begin{equation}
  \begin{pmatrix}
    l_1 \\
    l_2
  \end{pmatrix}
  \begin{pmatrix}
    r_1 & r_2 \\
  \end{pmatrix}
  =
  \begin{pmatrix}
    1 & 0 \\
    0 & 1
  \end{pmatrix},
\end{equation}
we can calculate the component $u_i$ by
\begin{equation}
  u_i=l_i\bm{u}.
\end{equation}

Next, we introduce \textit{diffusion waves} as in~\cite{Koike21,LZ97}. Let
\begin{equation}
  \label{eq:Mi}
  M_i=\int_{-\infty}^{\infty}u_{0i}(x)\, dx+l_i
  \begin{pmatrix}
    0 \\
    V_{0}^{0}
  \end{pmatrix}
  +l_i
  \begin{pmatrix}
    0 \\
    V_{1}^{0}
  \end{pmatrix},
\end{equation}
where
\begin{equation}
  u_{0i}=l_i
  \begin{pmatrix}
    v_0-1 \\
    u_0
  \end{pmatrix}.
\end{equation}
Then the $i$-th diffusion wave with mass $M_i$ is defined as the solution $\theta_i$ to generalized Burgers' equation
\begin{equation}
  \label{eq:theta}
  \partial_t \theta_i+\lambda_i \partial_x \theta_i+\partial_x \left( \frac{\theta_{i}^{2}}{2} \right)=\frac{\nu}{2}\partial_{x}^{2}\theta_i, \quad x\in \mathbb{R},\, t>0
\end{equation}
with the initial condition
\begin{equation}
  \label{eq:theta_init}
  \lim_{t\to -1}\theta_i(x,t)=M_i \delta(x).
\end{equation}
Here, $\delta(x)$ is the Dirac delta function. By the Cole--Hopf transformation, we can solve~\eqref{eq:theta} and~\eqref{eq:theta_init} explicitly to obtain
\begin{equation}
  \label{eq:theta_explicit}
  \theta_i(x,t)=\frac{\sqrt{\nu}}{\sqrt{2(t+1)}}\left( e^{\frac{M_i}{\nu}}-1 \right) e^{-\frac{(x-\lambda_i(t+1))^2}{2\nu(t+1)}}\left[ \sqrt{\pi}+\left( e^{\frac{M_i}{\nu}}-1 \right) \int_{\frac{x-\lambda_i(t+1)}{\sqrt{2\nu(t+1)}}}^{\infty}e^{-y^2}\, dy \right]^{-1}.
\end{equation}

Next, we introduce \textit{bi-diffusion waves} as in~\cite{Koike20-p}: the $i$-th bi-diffusion wave with mass pair $(M_1,M_2)$ is defined as the solution $\xi_i$ to the following variable coefficient inhomogeneous convective heat equation:
\begin{equation}
  \label{eq:xi}
  \partial_t \xi_i +\lambda_i \partial_x \xi_i +\partial_x (\theta_i \xi_i)+\partial_x \left( \frac{\theta_{i'}^{2}}{2} \right)=\frac{\nu}{2}\partial_{x}^{2}\xi_i, \quad x\in \mathbb{R},t>0
\end{equation}
with the initial condition
\begin{equation}
  \label{eq:xi_init}
  \xi_i(x,0)=0, \quad x\in \mathbb{R}.
\end{equation}
Here, $i'=3-i$, i.e., $1'=2$ and $2'=1$.

We next define some auxiliary functions. First, let
\begin{align}
  \psi_{\alpha}(x,t;\lambda_i) & =[(x-\lambda_i(t+1))^2+(t+1)]^{-\alpha/2}, \\
  \tilde{\psi}(x,t;\lambda_i) & =[|x-\lambda_i(t+1)|^3+(t+1)^2]^{-1/2}, \\
  \bar{\psi}(x,t;\lambda_i) & =[|x-\lambda_i(t+1)|^7+(t+1)^5]^{-1/4},
\end{align}
and
\begin{align}
  \Phi_i(x,t) & =\psi_{3/2}(x,t;\lambda_i)+\tilde{\psi}(x,t;\lambda_{i'}), \\
  \Psi_i(x,t) & =\psi_{7/4}(x,t;\lambda_i)+\bar{\psi}(x,t;\lambda_{i'}).
\end{align}
Moreover, the following functions are needed to state the compatibility conditions:
\begin{align}
  \mathcal{C}_1(v,u) & \coloneqq -p(v)+\nu \frac{u_x}{v}, \\
  \mathcal{C}_2(v,u) & \coloneqq -p'(v)u_x+\frac{\nu}{v}\mathcal{C}_1(v,u)_{xx}-\nu \frac{u_{x}^{2}}{v^2}.
\end{align}
Also, let
\begin{equation}
  u_{0i}^{-}(x)\coloneqq \int_{-\infty}^{x}u_{0i}(y)\, dy, \quad u_{0i}^{+}(x)\coloneqq \int_{x}^{\infty}u_{0i}(y)\, dy.
\end{equation}

Finally, let $\llbracket f \rrbracket(x)\coloneqq f(x_+)-f(x_-)$ and denote by $||\cdot ||_{k}$ ($k\in \mathbb{N}$) the Sobolev $H^k(\mathbb{R}_*)$-norm.

\subsection{Pointwise estimates of solutions and the long-time behavior of the point particles}
The first of our main theorems is the following, which is an extension of~\cite[Theorem~1.2]{Koike21}.

\begin{thm}
  \label{thm:main_H4}
  Let $v_0-1,u_0 \in H^4(\mathbb{R}_*)$ and $V_{0}^{0},V_{1}^{0}\in \mathbb{R}$. Assume that they satisfy the following compatibility conditions:
  \begin{align}
    u_0(0_{\pm})=V_{0}^{0} & , \quad u_0(1_{\pm})=V_{1}^{0}, \\
    \mathcal{C}_1(v_0,u_0)_x(0_{\pm})=\llbracket \mathcal{C}_1(v_0,u_0) \rrbracket(0), &  \quad \mathcal{C}_1(v_0,u_0)_x(1_{\pm})=\llbracket \mathcal{C}_1(v_0,u_0) \rrbracket(1).
  \end{align}
  Under these assumptions, there exist $\delta_0',C>0$ such that if
  \begin{align}
    \label{thm:main_H4:eq:delta_H4}
    \begin{aligned}
      \delta'
      & \coloneqq \sum_{i=1}^{2}\biggl\{ ||u_{0i}||_4+||u_{0i}^{-}||_{L^1(-\infty,0)}+||u_{0i}^{+}||_{L^1(0,\infty)} \\
      & \phantom{\coloneqq \sum_{i=1}^{2}\biggl[}
      +\sup_{x\in \mathbb{R}^*}\left[ (|x|+1)^{3/2}|u_{0i}(x)| \right]+\sup_{x>0}\left[ (|x|+1)(|u_{0i}^{-}(-x)|+|u_{0i}^{+}(x)|) \right] \biggr\} \leq \delta_0',
    \end{aligned}
  \end{align}
  then~\eqref{eq:FundEq} has a unique global-in-time solution $(v,u,V_0,V_1)$ satisfying
  \begin{align}
    v-1 & \in C([0,\infty);H^4(\mathbb{R}_*))\cap C^1([0,\infty);H^3(\mathbb{R}_*)), \\
    u   & \in C([0,\infty);H^4(\mathbb{R}_*))\cap C^1([0,\infty);H^2(\mathbb{R}_*)), \\
    u_x & \in L^2(0,\infty;H^4(\mathbb{R}_*)), \\
    V_0,V_1 & \in C^2([0,\infty))
  \end{align}
  and
  \begin{equation}
    ||(v-1)(t)||_4+||u(t)||_4+\left( \int_{0}^{\infty}||u_x(s)||_{4}^{2}\, ds \right)^{1/2}+\sum_{i=0}^{1}\sum_{k=0}^{2}|\partial_{t}^{k}V_i(t)|\leq C\delta' \quad (t\geq 0).
  \end{equation}
  Moreover, this solution satisfies the following pointwise estimates:
  \begin{equation}
    \label{thm:main_H4:eq:pwe_H4}
    |(u_i-\theta_i)(x,t)|\leq C\delta' \Phi_i(x,t) \quad (x\in \mathbb{R}_*,t\geq 0; i=1,2).
  \end{equation}
\end{thm}

From the theorem above, we obtain a decay estimate for $V_i(t)$.

\begin{cor}
  \label{cor:V_upper}
  Under the assumptions of Theorem~\ref{thm:main_H4}, there exist $\delta_0',C>0$ such that if~\eqref{thm:main_H4:eq:delta_H4} holds, then the solution $(v,u,V_0,V_1)$ to~\eqref{eq:FundEq} satisfies
  \begin{equation}
    \label{cor:V_upper:eq:V_upper}
    |V_i(t)|\leq C\delta'(t+1)^{-3/2} \quad (t\geq 0;i=0,1).
  \end{equation}
\end{cor}

With some additional assumptions on the regularity and the spatial decay of initial data, we can obtain finer pointwise estimates corresponding to~\cite[Theorem~2.1]{Koike20-p}.

\begin{thm}
  \label{thm:main}
  Let $v_0-1,u_0 \in H^6(\mathbb{R}_*)$ and $V_{0}^{0},V_{1}^{0}\in \mathbb{R}$. Assume that they satisfy the following compatibility conditions:
  \begin{align}
    u_0(0_{\pm})=V_{0}^{0} & , \quad u_0(1_{\pm})=V_{1}^{0}, \\
    \mathcal{C}_1(v_0,u_0)_x(0_{\pm})=\llbracket \mathcal{C}_1(v_0,u_0) \rrbracket(0), &  \quad \mathcal{C}_1(v_0,u_0)_x(1_{\pm})=\llbracket \mathcal{C}_1(v_0,u_0) \rrbracket(1), \\
    \mathcal{C}_2(v_0,u_0)_x(0_{\pm})=\llbracket \mathcal{C}_2(v_0,u_0) \rrbracket(0), &  \quad \mathcal{C}_2(v_0,u_0)_x(1_{\pm})=\llbracket \mathcal{C}_2(v_0,u_0) \rrbracket(1).
  \end{align}
  Under these assumptions, there exist $\delta_0,C>0$ such that if
  \begin{equation}
    \label{thm:main:eq:delta}
    \delta \coloneqq \sum_{i=1}^{2}\left\{ ||u_{0i}||_6+\sup_{x\in \mathbb{R}^*}\left[ (|x|+1)^{7/4}|u_{0i}(x)| \right]+\sup_{x>0}\left[ (|x|+1)^{5/4}(|u_{0i}^{-}(-x)|+|u_{0i}^{+}(x)|) \right] \right\} \leq \delta_0,
  \end{equation}
  then~\eqref{eq:FundEq} has a unique global-in-time solution $(v,u,V_0,V_1)$ satisfying
  \begin{align}
    v-1 & \in C([0,\infty);H^6(\mathbb{R}_*))\cap C^1([0,\infty);H^5(\mathbb{R}_*)), \\
    u   & \in C([0,\infty);H^6(\mathbb{R}_*))\cap C^1([0,\infty);H^4(\mathbb{R}_*)), \\
    u_x & \in L^2(0,\infty;H^6(\mathbb{R}_*)), \\
    V_0,V_1 & \in C^3([0,\infty))
  \end{align}
  and
  \begin{equation}
    ||(v-1)(t)||_6+||u(t)||_6+\left( \int_{0}^{\infty}||u_x(s)||_{6}^{2}\, ds \right)^{1/2}+\sum_{i=0}^{1}\sum_{k=0}^{3}|\partial_{t}^{k}V_i(t)|\leq C\delta \quad (t\geq 0).
  \end{equation}
  Moreover, this solution satisfies the following pointwise estimates:
  \begin{equation}
    \label{thm:main:eq:pwe}
    |(u_i-\theta_i-\xi_i-\gamma_{i'}\partial_x \theta_{i'})(x,t)|\leq C\delta \Psi_i(x,t) \quad (x\in \mathbb{R}_*,t\geq 0; i=1,2),
  \end{equation}
  where $i'=3-i$ and $\gamma_i=(-1)^i \nu/(4c)$.
\end{thm}

From this theorem, we obtain a simple necessary and sufficient condition for the optimality of the decay estimate $V_i(t)=O(t^{-3/2})$ given by Corollary~\ref{cor:V_upper}.

\begin{cor}
  \label{cor:V_lower}
  Define $M_i$ by~\eqref{eq:Mi} and assume that $(M_1+M_2)(M_1-M_2)\neq 0$, that is,
  \begin{equation}
    \label{cor:V_lower:eq:ass_nonzero}
    \left( \int_{-\infty}^{\infty}(v_0-1)(x)\, dx \right) \cdot \left( \int_{-\infty}^{\infty}u_0(x)\, dx+V_{0}^{0}+V_{1}^{0} \right) \neq 0.
  \end{equation}
  Then under the assumptions of Theorem~\ref{thm:main}, there exist $\delta_0>0$, $C>1$, and $T(\delta)>0$ such that if~\eqref{thm:main:eq:delta} holds, then the solution $(v,u,V_0,V_1)$ to~\eqref{eq:FundEq} satisfies
  \begin{equation}
    \label{cor:V_lower:eq:V_lower}
    C^{-1}|M_{1}^{2}-M_{2}^{2}|(t+1)^{-3/2}\leq (\sgn(M_{1}^{2}-M_{2}^{2}))V_i(t)\leq C|M_{1}^{2}-M_{2}^{2}|(t+1)^{-3/2} \quad (t\geq T(\delta);i=0,1).
  \end{equation}
  In particular, this implies
  \begin{equation}
    C^{-1}|M_{1}^{2}-M_{2}^{2}|(t+1)^{-3/2}\leq |V_i(t)|\leq C|M_{1}^{2}-M_{2}^{2}|(t+1)^{-3/2} \quad (t\geq T(\delta);i=0,1).
  \end{equation}
\end{cor}

\begin{cor}
  \label{cor:V_improved_decay}
  Define $M_i$ by~\eqref{eq:Mi} and assume that $(M_1+M_2)(M_1-M_2)=0$, that is,
  \begin{equation}
    \label{cor:V_lower:eq:ass_zero}
    \left( \int_{-\infty}^{\infty}(v_0-1)(x)\, dx \right) \cdot \left( \int_{-\infty}^{\infty}u_0(x)\, dx+V_{0}^{0}+V_{1}^{0} \right)=0.
  \end{equation}
  Then under the assumptions of Theorem~\ref{thm:main}, there exist $\delta_0,C>0$ such that if~\eqref{thm:main:eq:delta} holds, then the solution $(v,u,V_0,V_1)$ to~\eqref{eq:FundEq} satisfies
  \begin{equation}
    \label{cor:V_improved_decay:eq:V_improved_decay}
    |V_i(t)|\leq C\delta(t+1)^{-7/4} \quad (t\geq 0;i=0,1).
  \end{equation}
\end{cor}

\subsection{Discussion}
\subsubsection{Long-time behavior of the point particles}
Corollary~\ref{cor:V_upper} shows that, in general, we have $V(t)=O(t^{-3/2})$. If, in addition, we assume that~\eqref{cor:V_lower:eq:ass_nonzero} holds, which is equivalent to
\begin{equation}
  \int_{-\infty}^{\infty}\rho_0(X)\, dX\neq 0, \quad \int_{-\infty}^{\infty}(\rho_0 U_0)(X)\, dX+V_{0}^{0}+V_{1}^{0}\neq 0
\end{equation}
in the Eulerian coordinate, Corollary~\ref{cor:V_lower} shows that the decay rate $-3/2$ is optimal. Moreover, Corollary~\ref{cor:V_improved_decay} tells us that the condition above is a necessary and sufficient condition for the optimality of the decay rate $-3/2$.

These conclusions are similar to the ones for a single point particle~\cite{Koike20-p,Koike21}; the novelty of this paper lies in the proof, which we shall present in Section~\ref{sec:proofs}.

\subsubsection{Lack of collisions}
\label{subsec:NoCollisions}
Theorems~\ref{thm:main_H4} and~\ref{thm:main} imply that the density $\rho=1/v$ is bounded uniformly in time. Then noting the conservation of mass
\begin{equation}
  \int_{h_{0}(t)}^{h_{1}(t)}\rho(X,t)\, dX=1,
\end{equation}
we obtain
\begin{equation}
  h_{1}(t)-h_{0}(t)\geq \frac{1}{\sup_{t\geq 0,\, X\in (h_0(t),h_1(t))}\rho(X,t)}>0,
\end{equation}
which says that the two particles do not collide. Unlike viscous Burgers' equation considered in~\cite{VZ06}, due to the presence of the pressure (the density), the possibility of collisions between particles can easily be eliminated.

\section{Proofs}
\label{sec:proofs}
The basic strategy of the proof is identical to that of~\cite[Theorem~1.2]{Koike21} and~\cite[Theorem~2.1]{Koike20-p} and consists of three steps. Step~(i): Derive integral equations satisfied by the solution, and from there, define ``Green's functions'' associated to~\eqref{eq:FundEq}; Step~(ii): Obtain sharp enough pointwise estimates of Green's functions; and Step~(iii): Use the preceding two Steps to conduct nonlinear estimates. The presence of several point particles makes Steps~(i) and (ii) more complicated than that for a single point particle, but Step (iii) is almost identical. For this reason, we concentrate our attention on Steps~(i) and (ii), and Step~(iii) is only touched upon briefly; also, since the global-in-time existence parts of Theorems~\ref{thm:main_H4} and~\ref{thm:main} can be proved similarly to~\cite[Theorem~1.1]{Koike21}, we omit their proof.

\subsection{Integral equations}
The first step of the proof is to derive integral equations satisfied by the solution $(v,u,V_1,V_2)$ to~\eqref{eq:FundEq}. The idea of the derivation is similar to that of~\cite[Proposition~3.1]{Koike20-p} and uses the Laplace transform techniques originally developed in~\cite{LY11,LY12} and was used also, for example, in~\cite{DW18}. However, the resulting formulae turn out to be quite complicated. Nevertheless, we can still give a nice physical interpretation (see Remark~\ref{rem:IntEqInterpretation}).

To write down the integral equations, we first define Green's functions using the Laplace transform. Let $\mathcal{L}$ be the Laplace transform in time $t$, and denote by $s$ the Laplace variable. Then, let
\begin{equation}
  \label{def:Cij}
  C_{i,j}=C_{i,j}(s)=\frac{s^j\left( 2\sqrt{\nu s+c^2} \right)^i}{(s+2\sqrt{\nu s+c^2})^{i+j}}
  \begin{pmatrix}
    1 & 0 \\
    0 & -1
  \end{pmatrix}^{j}
  =\frac{2^i \lambda^j}{(\lambda+2)^{i+j}}
  \begin{pmatrix}
    1 & 0 \\
    0 & -1
  \end{pmatrix}^{j} \quad (i\in \{ 0,1,2 \},j\in \mathbb{N}_{\geq 0}),
\end{equation}
where $\lambda=s/\sqrt{\nu s+c^2}$. Here, $\lambda$ is defined on $\mathbb{C}\backslash (-\infty,-c^2/\nu]$ and the branch of $\sqrt{\nu s+c^2}$ is chosen so that $\sqrt{\nu s+c^2}>0$ for $s>-c^2/\nu$. Next, we introduce the fundamental solution $G$ as the solution to the following equations: 
\begin{equation}
  \label{eq:FundamentalSolution}
  \begin{dcases}
    \partial_t G+
    \begin{pmatrix}
      0    & -1 \\
      -c^2 & 0
    \end{pmatrix}
    \partial_x G=
    \begin{pmatrix}
      0 & 0 \\
      0 & \nu
    \end{pmatrix}
    \partial_{x}^{2}G,   & x\in \mathbb{R},\, t>0, \\
    G(x,0)=\delta(x)I_2, & x\in \mathbb{R},
  \end{dcases}
\end{equation}
where $\delta(x)$ is the Dirac delta function and $I_2$ is the $2\times 2$ identity matrix. Using these, we define
\begin{equation}
  \label{def:Gij}
  G_{i,j}(x,t)=\mathcal{L}^{-1}[\mathcal{L}[G]C_{i,j}](x,t),
\end{equation}
where $\mathcal{L}^{-1}$ is the inverse Laplace transform (the Bromwich integral is taken along a vertical contour with $\Re s\geq -\sigma_0$, where $\sigma_0$ is the positive constant in Lemma~\ref{lem:LaplaceAsymptotics}). We note that $G=G_{0,0}$. Now, we define a bunch of ``Green's functions'':
\begin{equation}
  \label{eq:GreenFunc1}
  \begin{gathered}
    G_{++}(x,t)=G(x,t), \quad G^{++}(x,t)=G_{0,1}(x-2,t)+\sum_{i=0}^{\infty}G_{2,2i+1}(x+2i,t), \\
    G_{0+}(x,t)=\sum_{i=0}^{\infty}G_{1,2i}(x+2i,t), \quad G^{0+}(x,t)=\sum_{i=0}^{\infty}G_{1,2i+1}(x+2i,t), \\
    G_{-+}(x,t)=\sum_{i=0}^{\infty}G_{2,2i}(x+2i,t), \\
    G_{0b}^{+}(x,t)=G_{-+}(x,t), \quad G_{1b}^{+}(x,t)=G_{0+}(x-1,t)+G^{0+}(x+1,t)
  \end{gathered}
\end{equation}
and
\begin{equation}
  \label{eq:GreenFunc2}
  \begin{gathered}
    G_{+0}(x,t)=\sum_{i=0}^{\infty}G_{1,2i}(x-2i,t), \quad G^{+0}(x,t)=\sum_{i=0}^{\infty}G_{1,2i+1}(x+2i,t), \\
    G_{00}(x,t)=G(x,t)+\sum_{i=1}^{\infty}\left[ G_{0,2i}(x+2i,t)+G_{0,2i}(x-2i,t) \right], \\
    G^{00}(x,t)=\sum_{i=0}^{\infty}\left[ G_{0,2i+1}(x+2i,t)+G_{0,2i+1}(x-2(i+1),t) \right], \\
    G_{-0}(x,t)=\sum_{i=0}^{\infty}G_{1,2i}(x+2i,t), \quad G^{-0}(x,t)=\sum_{i=0}^{\infty}G_{1,2i+1}(x-2(i+1),t), \\
    G_{0b}^{0}(x,t)=G_{-0}(x,t)+G^{-0}(x,t), \quad G_{1b}^{0}(x,t)=G_{+0}(x-1,t)+G^{+0}(x+1,t).
  \end{gathered}
\end{equation}
We note that these infinite sums actually converge as we see in the proof of Proposition~\ref{prop:GreenPWE}. Moreover, these complicated functions have a nice physical interpretation (see Remark~\ref{rem:IntEqInterpretation}).

By using these Green's functions, we can write down integral equations for the solution to~\eqref{eq:FundEq}.

\begin{prop}
  \label{prop:IntEq}
  Let $(v,u,V_1,V_2)$ be the global-in-time solution to~\eqref{eq:FundEq}. Then it satisfies the following integral equations. Case (i) $x>1$:
  \begin{align}
    \label{prop:IntEq:eq:IntEq:Case1}
    \begin{aligned}
      \begin{pmatrix}
        v-1 \\
        u
      \end{pmatrix}
      (x,t)
      & =\int_{1}^{\infty}\left[ G_{++}(x-y,t)+G^{++}(x+y,t) \right]
      \begin{pmatrix}
        v_0-1 \\
        u_0
      \end{pmatrix}
      (y)\, dy \\
      & \quad +\int_{0}^{t}\int_{1}^{\infty}\left[ G_{++}(x-y,t-s)+G^{++}(x+y,t-s) \right]
      \begin{pmatrix}
        0 \\
        N_x
      \end{pmatrix}
      (y,s)\, dyds \\
      & \quad +\int_{0}^{1}\left[ G_{0+}(x-y,t)+G^{0+}(x+y,t) \right]
      \begin{pmatrix}
        v_0-1 \\
        u_0
      \end{pmatrix}
      (y)\, dy \\
      & \quad +\int_{0}^{t}\int_{0}^{1}\left[ G_{0+}(x-y,t-s)+G^{0+}(x+y,t-s) \right]
      \begin{pmatrix}
        0 \\
        N_x
      \end{pmatrix}
      (y,s)\, dyds \\
      & \quad +\int_{-\infty}^{0}G_{-+}(x-y,t)
      \begin{pmatrix}
        v_0-1 \\
        u_0
      \end{pmatrix}
      (y)\, dy+\int_{0}^{t}\int_{-\infty}^{0}G_{-+}(x-y,t-s)
      \begin{pmatrix}
        0 \\
        N_x
      \end{pmatrix}
      (y,s)\, dyds \\
      & \quad +G_{0b}^{+}(x,t)
      \begin{pmatrix}
        0 \\
        V_{0}^{0}
      \end{pmatrix}
      +\int_{0}^{t}G_{0b}^{+}(x,t-s)
      \begin{pmatrix}
        0 \\
        \llbracket N \rrbracket
      \end{pmatrix}
      (0,s)\, ds \\
      & \quad +G_{1b}^{+}(x,t)
      \begin{pmatrix}
        0 \\
        V_{1}^{0}
      \end{pmatrix}
      +\int_{0}^{t}G_{1b}^{+}(x,t-s)
      \begin{pmatrix}
        0 \\
        \llbracket N \rrbracket
      \end{pmatrix}
      (1,s)\, ds.
    \end{aligned}
  \end{align}
  Case (ii) $0<x<1$:
  \begin{align}
    \label{prop:IntEq:eq:IntEq:Case2}
    \begin{aligned}
      \begin{pmatrix}
        v-1 \\
        u
      \end{pmatrix}
      (x,t)
      & =\int_{1}^{\infty}\left[ G_{+0}(x-y,t)+G^{+0}(x+y,t) \right]
      \begin{pmatrix}
        v_0-1 \\
        u_0
      \end{pmatrix}
      (y)\, dy \\
      & \quad +\int_{0}^{t}\int_{1}^{\infty}\left[ G_{+0}(x-y,t-s)+G^{+0}(x+y,t-s) \right]
      \begin{pmatrix}
        0 \\
        N_x
      \end{pmatrix}
      (y,s)\, dyds \\
      & \quad +\int_{0}^{1}\left[ G_{00}(x-y,t)+G^{00}(x+y,t) \right]
      \begin{pmatrix}
        v_0-1 \\
        u_0
      \end{pmatrix}
      (y)\, dy \\
      & \quad +\int_{0}^{1}\left[ G_{00}(x-y,t-s)+G^{00}(x+y,t-s) \right]
      \begin{pmatrix}
        0 \\
        N_x
      \end{pmatrix}
      (y,s)\, dy \\
      & \quad +\int_{-\infty}^{0}\left[ G_{-0}(x-y,t)+G^{-0}(x+y,t) \right]
      \begin{pmatrix}
        v_0-1 \\
        u_0
      \end{pmatrix}
      (y)\, dy \\
      & \quad +\int_{0}^{t}\int_{-\infty}^{0}\left[ G_{-0}(x-y,t-s)+G^{-0}(x+y,t-s) \right]
      \begin{pmatrix}
        0 \\
        N_x
      \end{pmatrix}
      (y,s)\, dyds \\
      & \quad +G_{0b}^{0}(x,t)
      \begin{pmatrix}
        0 \\
        V_{0}^{0}
      \end{pmatrix}
      +\int_{0}^{t}G_{0b}^{0}(x,t-s)
      \begin{pmatrix}
        0 \\
        \llbracket N \rrbracket
      \end{pmatrix}
      (0,s)\, ds \\
      & \quad +G_{1b}^{0}(x,t)
      \begin{pmatrix}
        0 \\
        V_{1}^{0}
      \end{pmatrix}
      +\int_{0}^{t}G_{1b}^{0}(x,t-s)
      \begin{pmatrix}
        0 \\
        \llbracket N \rrbracket
      \end{pmatrix}
      (1,s)\, ds.
    \end{aligned}
  \end{align}
  Case (iii) $x<0$: a similar formula holds.
\end{prop}

\begin{proof}
  We only prove Case (i) since Case (ii) can be proved similarly (although the required computations are somewhat lengthier). Let $(v_1,u_1)$ be the solution to the following Cauchy problem with initial data $(v_0,u_0)$:
  \begin{equation}
    \begin{dcases}
      \partial_t v_1-\partial_x u_1=0, & x\in \mathbb{R},\, t>0, \\
      \partial_t u_1-c^2 \partial_x v_1=\nu \partial_{x}^{2}u_1+\partial_x N, & x\in \mathbb{R},\, t>0, \\
      v_1(x,0)=v_0(x),\, u_1(x,0)=u_0(x), & x\in \mathbb{R}.
    \end{dcases}
  \end{equation}
  Note that the nonlinear term $N$ is defined in~\eqref{eq:vectors_matrices} using $(v,u)$ and not $(v_1,u_1)$. Note also that since the initial data may have discontinuity at $x=0$, the solution is considered in the generalized sense. In any case, it can be represented by using the fundamental solution $G$ as follows:
  \begin{equation}
    \label{prop:IntEq:proof:eq:IntRepCauchy}
    \begin{pmatrix}
      v_1-1 \\
      u_1
    \end{pmatrix}
    (x,t)=\int_{-\infty}^{\infty}G(x-y,t)
    \begin{pmatrix}
      v_0-1 \\
      u_0
    \end{pmatrix}
    (y)\, dy+\int_{0}^{t}\int_{-\infty}^{\infty}G(x-y,t-s)
    \begin{pmatrix}
      0 \\
      N_x
    \end{pmatrix}
    (y,s)\, dyds.
  \end{equation}

  Next, let
  \begin{equation}
    (v_2,u_2)\coloneqq (v-v_1,u-u_1).
  \end{equation}
  Then we have (see~\cite[p.~378--379]{Koike21})
  \begin{equation}
    \label{prop:IntEq:proof:eq:LaplaceTransformedEq}
    \begin{dcases}
      s\hat{v}_2-\partial_x \hat{u}_2=0, \\
      s\hat{u}_2-c^2 \partial_x \hat{v}_2=\nu \partial_{x}^{2}\hat{u}_2, \\
      s\hat{u}_2(0_{\pm},s)-(\nu s+c^2)\llbracket \hat{v}_2 \rrbracket(0,s)=\llbracket \hat{N} \rrbracket(0,s)-s\hat{u}_1(0,s)+V_{0}^{0}, \\
      s\hat{u}_2(1_{\pm},s)-(\nu s+c^2)\llbracket \hat{v}_2 \rrbracket(1,s)=\llbracket \hat{N} \rrbracket(1,s)-s\hat{u}_1(1,s)+V_{1}^{0},
    \end{dcases}
  \end{equation}
  where the variables with a hat are the Laplace transformed variables. General solutions to~\eqref{prop:IntEq:proof:eq:LaplaceTransformedEq} have the form
  \begin{equation}
    \label{prop:IntEq:proof:eq:GeneralSolutions}
    \begin{pmatrix}
      \hat{v}_2 \\
      \hat{u}_2
    \end{pmatrix}
    (x,s)=
    \begin{dcases}
      C_+
      \begin{pmatrix}
        -\lambda/s \\
        1
      \end{pmatrix}
      e^{-\lambda x} & (x>1), \\
      C_0
      \begin{pmatrix}
        -\lambda/s \\
        1
      \end{pmatrix}
      e^{-\lambda x}+C_1
      \begin{pmatrix}
        \lambda/s \\
        1
      \end{pmatrix}
      e^{\lambda x} & (0<x<1), \\
      C_-
      \begin{pmatrix}
        \lambda/s \\
        1
      \end{pmatrix}
      e^{\lambda x} & (x<0).
    \end{dcases}
  \end{equation}
  We remind the reader that $\lambda=s/\sqrt{\nu s+c^2}$. Set
  \begin{equation}
    \label{prop:IntEq:proof:def:Psi0_Psi1}
    \Psi_0(s)=\llbracket \hat{N} \rrbracket(0,s)-s\hat{u}_1(0,s)+V_{0}^{0}, \quad \Psi_1(s)=\llbracket \hat{N} \rrbracket(1,s)-s\hat{u}_1(1,s)+V_{1}^{0}.
  \end{equation}
  By the third and the fourth equations in~\eqref{prop:IntEq:proof:eq:LaplaceTransformedEq}, we obtain the following equations that determine the constants $C_+$, $C_0$, $C_1$, and $C_-$:
  \begin{equation}
    \begin{dcases}
      sC_+ e^{-\lambda}-\sqrt{\nu s+c^2}(-C_+ e^{-\lambda}+C_0 e^{-\lambda}-C_1 e^{\lambda}) & =\Psi_1(s), \\
      sC_0 e^{-\lambda}+sC_1 e^{\lambda}-\sqrt{\nu s+c^2}(-C_+ e^{-\lambda}+C_0 e^{-\lambda}-C_1 e^{\lambda}) & =\Psi_1(s), \\
      sC_0+sC_1-\sqrt{\nu s+c^2}(-C_0+C_1-C_-) & =\Psi_0(s), \\
      sC_{-}-\sqrt{\nu s+c^2}(-C_0+C_1-C_-) & =\Psi_0(s).
    \end{dcases}
  \end{equation}
  Solving these equations, we obtain
  \begin{equation}
    \begin{dcases}
      C_0=\frac{-(s+2\sqrt{\nu s+c^2})e^{2\lambda}\Psi_0(s)+se^{\lambda}\Psi_1(s)}{[s+(s+2\sqrt{\nu s+c^2})e^{\lambda}][s-(s+2\sqrt{\nu s+c^2})e^{\lambda}]}, \\
      C_1=\frac{s\Psi_0(s)-(s+2\sqrt{\nu s+c^2})e^{\lambda}\Psi_1(s)}{[s+(s+2\sqrt{\nu s+c^2})e^{\lambda}][s-(s+2\sqrt{\nu s+c^2})e^{\lambda}]}, \\
      C_{+}=\frac{-2\sqrt{\nu s+c^2}e^{2\lambda}\Psi_0(s)+[s-(s+2\sqrt{\nu s + c^2})e^{2\lambda}]e^{\lambda}\Psi_1(s)}{[s+(s+2\sqrt{\nu s+c^2})e^{\lambda}][s-(s+2\sqrt{\nu s+c^2})e^{\lambda}]}, \\
      C_{-}=\frac{[s-(s+2\sqrt{\nu s+c^2})e^{2\lambda}]\Psi_0(s)-2\sqrt{\nu s+c^2}e^{\lambda}\Psi_1(s)}{[s+(s+2\sqrt{\nu s+c^2})e^{\lambda}][s-(s+2\sqrt{\nu s+c^2})e^{\lambda}]}.
    \end{dcases}
  \end{equation}
  Set
  \begin{equation}
    \label{def:r}
    r=r(s)=\frac{s^2}{(s+2\sqrt{\nu s+c^2})^2}=\frac{\lambda^2}{(\lambda+2)^2}.
  \end{equation}
  Then
  \begin{equation}
    \frac{e^{2\lambda}}{[s+(s+2\sqrt{\nu s+c^2})e^{\lambda}][s-(s+2\sqrt{\nu s+c^2})e^{\lambda}]}=\frac{-1}{(s+2\sqrt{\nu s+c^2})^2(1-re^{-2\lambda})}.
  \end{equation}
  By Lemma~\ref{lem:LaplaceAsymptotics}, there exists $r_0 \in (0,1)$ such that $|re^{-2\lambda}|\leq r_0$ for all $s$ with $\Re s\geq -\sigma_0$; so we have
  \begin{equation}
    \frac{1}{1-re^{-2\lambda}}=\sum_{i=0}^{\infty}r^i e^{-2i\lambda}.
  \end{equation}
  Using this series expansion, we obtain
  \begin{align}
    \label{prop:IntEq:proof:eq:sC+}
    \begin{aligned}
      sC_{+}
      & =\sum_{i=0}^{\infty}c_{1,2i+1}e^{-2i\lambda}\Psi_0(s)-\sum_{i=0}^{\infty}c_{0,2i+2}e^{-(2i+1)\lambda}\Psi_1(s)+\sum_{i=0}^{\infty}c_{0,2i+1}e^{-(2i-1)\lambda}\Psi_1(s) \\
      & =\sum_{i=0}^{\infty}c_{1,2i+1}e^{-2i\lambda}\Psi_0(s)+c_{0,1}e^{\lambda}\Psi_1(s)+\sum_{i=1}^{\infty}(c_{0,2i+1}-c_{0,2i})e^{-(2i-1)\lambda}\Psi_1(s) \\
      & =\sum_{i=0}^{\infty}c_{1,2i+1}e^{-2i\lambda}\Psi_0(s)+\left( c_{0,1}-\sum_{i=1}^{\infty}c_{1,2i}e^{-2i\lambda} \right) e^{\lambda}\Psi_1(s),
    \end{aligned}
  \end{align}
  where
  \begin{equation}
    c_{i,j}=c_{i,j}(s)=\frac{s^j\left( 2\sqrt{\nu s+c^2} \right)^i}{(s+2\sqrt{\nu s+c^2})^{i+j}}=\frac{2^i \lambda^j}{(\lambda+2)^{i+j}}.
  \end{equation}
  Here, we used the formula
  \begin{equation}
    \label{prop:IntEq:proof:eq:cij_rel}
    c_{i,j}-c_{i,j+1}=c_{i+1,j}.
  \end{equation}

  We now recall~\cite[Eq.~(37)]{Koike21}:
  \begin{equation}
    \label{prop:IntEq:proof:eq:LaplaceTransformedG}
    \mathcal{L}[G](x,s)=\frac{1}{\nu s+c^2}
    \begin{pmatrix}
      \nu \delta(x)+\frac{c^2}{2\sqrt{\nu s+c^2}}e^{-\lambda |x|} & -\frac{\sgn(x)}{2}e^{-\lambda |x|} \\
      -\frac{c^2 \sgn(x)}{2}e^{-\lambda |x|}                      & \frac{s}{2\lambda}e^{-\lambda |x|}
    \end{pmatrix}.
  \end{equation}
  Also note that $\hat{u}_1(0,s)$ and $\hat{u}_1(1,s)$ in the definitions of $\Psi_0(s)$ and $\Psi_1(s)$, see~\eqref{prop:IntEq:proof:def:Psi0_Psi1}, can be expressed in terms of $\hat{G}$ using~\eqref{prop:IntEq:proof:eq:IntRepCauchy}. Then,~\eqref{prop:IntEq:proof:eq:GeneralSolutions},~\eqref{prop:IntEq:proof:eq:sC+},~and~\eqref{prop:IntEq:proof:eq:LaplaceTransformedG} lead us to the following formula for $x>1$:
  \begin{align}
    \begin{pmatrix}
      \hat{v}_2 \\
      \hat{u}_2
    \end{pmatrix}
    (x,s)
    & =
    \begin{pmatrix} 
      -\lambda/s \\
      1
    \end{pmatrix}
    \begin{pmatrix}
      0 & -1
    \end{pmatrix}
    e^{-\lambda x}\left( c_{0,1}-\sum_{i=1}^{\infty}c_{1,2i}e^{-2i\lambda} \right)e^{\lambda}\int_{-\infty}^{\infty}\hat{G}(1-y,s)
    \begin{pmatrix}
      (v_0-1)(y) \\
      u_0(y)+\hat{N}_x(y,s)
    \end{pmatrix}
    \, dy \\
    & \quad +
    \begin{pmatrix}
      -\lambda /s \\
      1
    \end{pmatrix}
    \frac{e^{-\lambda x}}{s}\left( c_{0,1}-\sum_{i=1}^{\infty}c_{1,2i}e^{-2i\lambda} \right)e^{\lambda}(V_{1}^{0}+\llbracket \hat{N} \rrbracket(1,s)) \\
    & \quad +
    \begin{pmatrix} 
      -\lambda/s \\
      1
    \end{pmatrix}
    \begin{pmatrix}
      0 & -1
    \end{pmatrix}
    e^{-\lambda x}\sum_{i=0}^{\infty}c_{1,2i+1}e^{-2i\lambda}\int_{-\infty}^{\infty}\hat{G}(0-y,s)
    \begin{pmatrix}
      (v_0-1)(y) \\
      u_0(y)+\hat{N}_x(y,s)
    \end{pmatrix}
    \, dy \\
    & \quad +
    \begin{pmatrix}
      -\lambda /s \\
      1
    \end{pmatrix}
    \frac{e^{-\lambda x}}{s}\sum_{i=0}^{\infty}c_{1,2i+1}e^{-2i\lambda}(V_{0}^{0}+\llbracket \hat{N} \rrbracket(0,s)) \\
    & =\int_{-\infty}^{\infty}\frac{e^{-\lambda(x-1+|y-1|)}}{2(\nu s+c^2)}
    \begin{pmatrix}
      -c^2 \sgn(1-y)\lambda/s & 1 \\
      c^2 \sgn(1-y) & -s/\lambda
    \end{pmatrix}
    \left( c_{0,1}-\sum_{i=1}^{\infty}c_{1,2i}e^{-2i\lambda} \right)
    \begin{pmatrix}
      (v_0-1)(y) \\
      u_0(y)+\hat{N}_x(y,s)
    \end{pmatrix}
    \, dy \\
    & \quad +
    \begin{pmatrix}
      -\lambda /s \\
      1
    \end{pmatrix}
    \frac{e^{-\lambda x}}{s}\left( c_{0,1}-\sum_{i=1}^{\infty}c_{1,2i}e^{-2i\lambda} \right)e^{\lambda}(V_{1}^{0}+\llbracket \hat{N} \rrbracket(1,s)) \\
    & \quad +\int_{-\infty}^{\infty}\frac{e^{-\lambda(x+|y|)}}{2(\nu s+c^2)}
    \begin{pmatrix}
      -c^2 \sgn(-y)\lambda/s & 1 \\
      c^2 \sgn(-y) & -s/\lambda
    \end{pmatrix}
    \sum_{i=0}^{\infty}c_{1,2i+1}e^{-2i\lambda}
    \begin{pmatrix}
      (v_0-1)(y) \\
      u_0(y)+\hat{N}_x(y,s)
    \end{pmatrix}
    \, dy \\
    & \quad +
    \begin{pmatrix}
      -\lambda /s \\
      1
    \end{pmatrix}
    \frac{e^{-\lambda x}}{s}\sum_{i=0}^{\infty}c_{1,2i+1}e^{-2i\lambda}(V_{0}^{0}+\llbracket \hat{N} \rrbracket(0,s)).
  \end{align}
  Taking into account~\eqref{prop:IntEq:proof:eq:LaplaceTransformedG} again, we obtain
  \begin{align}
    \label{prop:IntEq:proof:eq:IntEqv2u2}
    \begin{aligned}
      \begin{pmatrix}
        \hat{v}_2 \\
        \hat{u}_2
      \end{pmatrix}
      (x,s)
      & =\int_{1}^{\infty}\hat{G}(x+y-2,s)\left( c_{0,1}-\sum_{i=1}^{\infty}c_{1,2i}e^{-2i\lambda} \right)
      \begin{pmatrix}
        (v_0-1)(y) \\
        -u_0(y)-\hat{N}_x(y,s)
      \end{pmatrix}
      \, dy \\
      & \quad -\int_{-\infty}^{1}\hat{G}(x-y,s)\left( c_{0,1}-\sum_{i=1}^{\infty}c_{1,2i}e^{-2i\lambda} \right)
      \begin{pmatrix}
        (v_0-1)(y) \\
        u_0(y)+\hat{N}_x(y,s)
      \end{pmatrix}
      \, dy \\
      & \quad +\hat{G}(x-1,s)\left( c_{1,0}-\sum_{i=1}^{\infty}c_{2,2i-1}e^{-2i\lambda} \right)
      \begin{pmatrix}
        0 \\
        V_{1}^{0}+\llbracket \hat{N} \rrbracket(1,s)
      \end{pmatrix} \\
      & \quad +\int_{0}^{\infty}\hat{G}(x+y,s)\sum_{i=0}^{\infty}c_{1,2i+1}e^{-2i\lambda}
      \begin{pmatrix}
        (v_0-1)(y) \\
        -u_0(y)-\hat{N}_x(y,s)
      \end{pmatrix}
      \, dy \\
      & \quad -\int_{-\infty}^{0}\hat{G}(x-y,s)\sum_{i=0}^{\infty}c_{1,2i+1}e^{-2i\lambda}
      \begin{pmatrix}
        (v_0-1)(y) \\
        u_0(y)+\hat{N}_x(y,s)
      \end{pmatrix}
      \, dy \\
      & \quad +\hat{G}(x,s)\sum_{i=0}^{\infty}c_{2,2i}e^{-2i\lambda}
      \begin{pmatrix}
        0 \\
        V_{0}^{0}+\llbracket \hat{N} \rrbracket(0,s)
      \end{pmatrix}.
    \end{aligned}
  \end{align}
  Finally, adding the Laplace transformed~\eqref{prop:IntEq:proof:eq:IntRepCauchy} to~\eqref{prop:IntEq:proof:eq:IntEqv2u2} and noting~\eqref{def:Gij} and~\eqref{prop:IntEq:proof:eq:cij_rel}, we get
  \begin{align}
    \label{prop:IntEq:proof:eq:LaplaceTransformedSol}
    \begin{aligned}
      & \mathcal{L}\left[
      \begin{pmatrix}
        v-1 \\
        u
      \end{pmatrix}
      \right](x,s) \\
      & =\int_{1}^{\infty}\left[ \hat{G}(x-y,s)+\hat{G}_{0,1}(x+y-2,s)+\sum_{i=0}^{\infty}\hat{G}_{2,2i+1}(x+y+2i,s) \right]
      \begin{pmatrix}
        (v_0-1)(y) \\
        u_0(y)+\hat{N}_x(y,s)
      \end{pmatrix}
      \, dy \\
      & \quad +\int_{0}^{1}\sum_{i=0}^{\infty}\left[ \hat{G}_{1,2i}(x-y+2i,s)+\hat{G}_{1,2i+1}(x+y+2i,s) \right]
      \begin{pmatrix}
        (v_0-1)(y) \\
        u_0(y)+\hat{N}_x(y,s)
      \end{pmatrix}
      \, dy \\
      & \quad +\int_{-\infty}^{0}\sum_{i=0}^{\infty}\hat{G}_{2,2i}(x-y+2i,s)
      \begin{pmatrix}
        (v_0-1)(y) \\
        u_0(y)+\hat{N}_x(y,s)
      \end{pmatrix}
      \, dy \\
      & \quad +\sum_{i=0}^{\infty}\hat{G}_{2,2i}(x+2i,s)
      \begin{pmatrix}
        0 \\
        V_{0}^{0}+\llbracket \hat{N} \rrbracket(0,s)
      \end{pmatrix} \\
      & \quad +\left[ \hat{G}_{1,0}(x-1,s)+\sum_{i=1}^{\infty}\hat{G}_{2,2i-1}(x+2i-1,s) \right]
      \begin{pmatrix}
        0 \\
        V_{1}^{0}+\llbracket \hat{N} \rrbracket(1,s)
      \end{pmatrix}.
    \end{aligned}
  \end{align}
  For the last term on the right-hand side, note that by~\eqref{prop:IntEq:proof:eq:cij_rel}, we can rewrite it as
  \begin{align}
    & \left[ \hat{G}_{1,0}(x-1,s)+\sum_{i=1}^{\infty}\hat{G}_{2,2i-1}(x+2i-1,s) \right]
    \begin{pmatrix}
      0 \\
      *
    \end{pmatrix} \\
    & =\left[ \hat{G}_{1,0}(x-1,s)+\sum_{i=1}^{\infty}\hat{G}_{1,2i}(x+2i-1,s)+\sum_{i=1}^{\infty}\hat{G}_{1,2i-1}(x+2i-1,s) \right]
    \begin{pmatrix}
      0 \\
      *
    \end{pmatrix} \\
    & =\sum_{i=0}^{\infty}\left[ \hat{G}_{1,2i}(x+2i-1,s)+\hat{G}_{1,2i+1}(x+2i+1,s) \right]
    \begin{pmatrix}
      0 \\
      *
    \end{pmatrix}
    =\left[ \hat{G}_{0+}(x-1,s)+\hat{G}^{0+}(x+1,s) \right]
    \begin{pmatrix}
      0 \\
      *
    \end{pmatrix},
  \end{align}
  where $*$ denotes an arbitrary number. Now taking the inverse Laplace transform of~\eqref{prop:IntEq:proof:eq:LaplaceTransformedSol}, we obtain~\eqref{prop:IntEq:eq:IntEq:Case1}. This ends the proof.
\end{proof}

\begin{rem}[Physical interpretation of Proposition~\ref{prop:IntEq}\footnote{The interpretation below is inspired by a closely related analysis for a heat equation with a conductivity having jumps which is used in an unpublished preprint by Tai-Ping Liu and Shih-Hsien Yu~\cite{LY_preprint}.}]
  \label{rem:IntEqInterpretation}
  When a Dirac delta input $\delta(\cdot-y)I_2$ is given, the response, i.e., $G(\cdot-y,\cdot)$, propagates to the left and to the right. When one of these waves hits a point particle, transmission and reflection occur. We set the following rules: multiply $C_{1,0}$ to the wave (in the Laplace variable) when there is a transmission and multiply $C_{0,1}$ when there is a reflection. Now, suppose that a wave experiences $i$-times of transmissions and $j$-times of reflections before reaching position $x$ at time $t$, and let $\phi(x,y;i,j)$ be the total length traversed by the wave (we set $\phi(x,y;i,j)$ to be negative when the wave reaches $x$ from the right). Then the resulting wave at $(x,t)$ should be $G_{i,j}(\phi(x,y;i,j),t)$.

  Now we can give a simple interpretation of Proposition~\ref{prop:IntEq}. Fix $x>1$ and let us consider the term
  \begin{equation}
    \int_{1}^{\infty}G^{++}(x+y,t)
    \begin{pmatrix}
      v_0-1 \\
      u_0
    \end{pmatrix}
    (y)\, dy=\int_{1}^{\infty}\left[ G_{0,1}(x+y-2,t)+\sum_{i=0}^{\infty}G_{2,2i+1}(x+y+2i,t) \right]
    \begin{pmatrix}
      v_0-1 \\
      u_0
    \end{pmatrix}
    (y)\, dy
  \end{equation}
  in~\eqref{prop:IntEq:eq:IntEq:Case1}, which can be interpreted as follows: the term involving $G_{0,1}(x+y-2,t)$ is the contribution of the wave that comes from $y>1$ and reflects at the point particle number 1 to reach $x$ (the total distance traversed is $x+y-2$); the term involving $G_{2,2i+1}(x+y+2i,t)$ is the contribution of the wave that comes from $y>1$ and transmits at the point particle number 1 to reach the point particle number 0 and then reflects $2i+1$ times between the two point particles and finally transmits at the point particle number 1 to reach $x$ (the total distance traversed is $x+y+2i$). All the other terms in~\eqref{prop:IntEq:eq:IntEq:Case1} and~\eqref{prop:IntEq:eq:IntEq:Case2} can be interpreted in this way. However, the interpretation of the boundary terms such as
  \begin{equation}
    G_{0b}^{+}(x,t)
    \begin{pmatrix}
      0 \\
      V_{0}^{0}
    \end{pmatrix}
  \end{equation}
  is a little bit more complicated. First, note that the momentum carried by the point particle number 0 is $V_{0}^{0}$. From this, we interpret that this point particle gives an input
  \begin{equation}
    \delta(x-\epsilon)
    \begin{pmatrix}
      0 \\
      V_{0}^{0}
    \end{pmatrix},
  \end{equation}
  where we take $\epsilon \searrow 0$ in the end (since the mass of the point particle does not propagate into the fluid, the first component of the input is zero). Then the corresponding response to this input should be
  \begin{equation}
    G_{0b}^{+}(x,t)
    \begin{pmatrix}
      0 \\
      V_{0}^{0}
    \end{pmatrix}
    =\lim_{\epsilon \searrow 0}G_{-+}(x-\epsilon,t)
    \begin{pmatrix}
      0 \\
      V_{0}^{0}
    \end{pmatrix}.
  \end{equation}
  All the other boundary terms can be interpreted in a similar manner.
\end{rem}

Let us give a corollary of Proposition~\ref{prop:IntEq}, which is just an application of integration by parts in $x$.

\begin{cor}
  \label{cor:IntEq}
  Let $(v,u,V_1,V_2)$ be the global-in-time solution to~\eqref{eq:FundEq}. Then it satisfies the following integral equations. Case (i) $x>1$:
  \begin{align}
    \label{cor:IntEq:eq:IntEq:Case1}
    \begin{aligned}
      \begin{pmatrix}
        v-1 \\
        u
      \end{pmatrix}
      (x,t)
      & =\int_{1}^{\infty}\left[ G_{++}(x-y,t)+G^{++}(x+y,t) \right]
      \begin{pmatrix}
        v_0-1 \\
        u_0
      \end{pmatrix}
      (y)\, dy \\
      & \quad +\int_{0}^{t}\int_{1}^{\infty}\partial_x \left[ G_{++}(x-y,t-s)-G^{++}(x+y,t-s) \right]
      \begin{pmatrix}
        0 \\
        N
      \end{pmatrix}
      (y,s)\, dyds \\
      & \quad +\int_{0}^{1}\left[ G_{0+}(x-y,t)+G^{0+}(x+y,t) \right]
      \begin{pmatrix}
        v_0-1 \\
        u_0
      \end{pmatrix}
      (y)\, dy \\
      & \quad +\int_{0}^{t}\int_{0}^{1}\partial_x \left[ G_{0+}(x-y,t-s)-G^{0+}(x+y,t-s) \right]
      \begin{pmatrix}
        0 \\
        N
      \end{pmatrix}
      (y,s)\, dyds \\
      & \quad +\int_{-\infty}^{0}G_{-+}(x-y,t)
      \begin{pmatrix}
        v_0-1 \\
        u_0
      \end{pmatrix}
      (y)\, dy+\int_{0}^{t}\int_{-\infty}^{0}\partial_x G_{-+}(x-y,t-s)
      \begin{pmatrix}
        0 \\
        N
      \end{pmatrix}
      (y,s)\, dyds \\
      & \quad +G_{0b}^{+}(x,t)
      \begin{pmatrix}
        0 \\
        V_{0}^{0}
      \end{pmatrix}
      +G_{1b}^{+}(x,t)
      \begin{pmatrix}
        0 \\
        V_{1}^{0}
      \end{pmatrix}.
    \end{aligned}
  \end{align}
  Case (ii) $0<x<1$:
  \begin{align}
    \label{cor:IntEq:eq:IntEq:Case2}
    \begin{aligned}
      \begin{pmatrix}
        v-1 \\
        u
      \end{pmatrix}
      (x,t)
      & =\int_{1}^{\infty}\left[ G_{+0}(x-y,t)+G^{+0}(x+y,t) \right]
      \begin{pmatrix}
        v_0-1 \\
        u_0
      \end{pmatrix}
      (y)\, dy \\
      & \quad +\int_{0}^{t}\int_{1}^{\infty}\partial_x \left[ G_{+0}(x-y,t-s)-G^{+0}(x+y,t-s) \right]
      \begin{pmatrix}
        0 \\
        N
      \end{pmatrix}
      (y,s)\, dyds \\
      & \quad +\int_{0}^{1}\left[ G_{00}(x-y,t)+G^{00}(x+y,t) \right]
      \begin{pmatrix}
        v_0-1 \\
        u_0
      \end{pmatrix}
      (y)\, dy \\
      & \quad +\int_{0}^{1}\partial_x \left[ G_{00}(x-y,t-s)-G^{00}(x+y,t-s) \right]
      \begin{pmatrix}
        0 \\
        N
      \end{pmatrix}
      (y,s)\, dy \\
      & \quad +\int_{-\infty}^{0}\left[ G_{-0}(x-y,t)+G^{-0}(x+y,t) \right]
      \begin{pmatrix}
        v_0-1 \\
        u_0
      \end{pmatrix}
      (y)\, dy \\
      & \quad +\int_{0}^{t}\int_{-\infty}^{0}\partial_x \left[ G_{-0}(x-y,t-s)-G^{-0}(x+y,t-s) \right]
      \begin{pmatrix}
        0 \\
        N
      \end{pmatrix}
      (y,s)\, dyds \\
      & \quad +G_{0b}^{0}(x,t)
      \begin{pmatrix}
        0 \\
        V_{0}^{0}
      \end{pmatrix}
      +G_{1b}^{0}(x,t)
      \begin{pmatrix}
        0 \\
        V_{1}^{0}
      \end{pmatrix}.
    \end{aligned}
  \end{align}
  Case (iii) $x<0$: a similar formula holds.
\end{cor}

\begin{proof}
  These are simple consequences of integration by parts applied to the formulae in Proposition~\ref{prop:IntEq}. For example, to show~\eqref{cor:IntEq:eq:IntEq:Case1}, it suffices to check that
  \begin{equation}
    I=-\left[ G_{0+}(x,t-s)+G^{0+}(x,t-s) \right]
    \begin{pmatrix}
      0 \\
      N
    \end{pmatrix}
    (0_+,s)+G_{-+}(x,t-s)
    \begin{pmatrix}
      0 \\
      N
    \end{pmatrix}
    (0_-,s)+G_{0b}^{+}(x,t-s)
    \begin{pmatrix}
      0 \\
      \llbracket N \rrbracket
    \end{pmatrix}
    (0,s)
  \end{equation}
  and
  \begin{align}
    II
    & =-\left[ G_{++}(x-1,t-s)+G^{++}(x+1,t-s) \right]
    \begin{pmatrix}
      0 \\
      N
    \end{pmatrix}
    (1_+,s) \\
    & \quad +\left[ G_{0+}(x-1,t-s)+G^{0+}(x+1,t-s) \right]
    \begin{pmatrix}
      0 \\
      N
    \end{pmatrix}
    (1_-,s)+G_{1b}^{+}(x,t-s)
    \begin{pmatrix}
      0 \\
      \llbracket N \rrbracket
    \end{pmatrix}
    (1,s)
  \end{align}
  vanish. For brevity, we shall only show that $I$ is zero. First, note that by~\eqref{prop:IntEq:proof:eq:cij_rel}, we have
  \begin{align}
    \left[ G_{0+}(x,t-s)+G^{0+}(x,t-s) \right]
    \begin{pmatrix}
      0 \\
      N
    \end{pmatrix}
    (0_+,s)
    & =\sum_{i=0}^{\infty}\left[ G_{1,2i}(x+2i,t-s)+G_{1,2i+1}(x+2i,t-s) \right]
    \begin{pmatrix}
      0 \\
      N
    \end{pmatrix}
    (0_+,s) \\
    & =\sum_{i=0}^{\infty}G_{2,2i}(x+2i,t-s)
    \begin{pmatrix}
      0 \\
      N
    \end{pmatrix}
    (0_+,s) \\
    & =G_{-+}(x,t-s)
    \begin{pmatrix}
      0 \\
      N
    \end{pmatrix}
    (0_+,s).
  \end{align}
  Note also that $G_{0b}^{+}=G_{-+}$ by definition. From these, it follows that
  \begin{equation}
    I=-G_{-+}(x,t-s)
    \begin{pmatrix}
      0 \\
      \llbracket N \rrbracket
    \end{pmatrix}
    (0,s)+G_{0b}^{+}(x,t-s)
    \begin{pmatrix}
      0 \\
      \llbracket N \rrbracket
    \end{pmatrix}
    (0,s)=0.
  \end{equation}
  We can show~\eqref{cor:IntEq:eq:IntEq:Case2} in a similar manner. This ends the proof.
\end{proof}

\subsection{Pointwise estimates of Green's functions}
\label{sec:PWE_GreensFunctions}
The next step of the proof is to obtain sharp enough pointwise estimates of Green's functions defined in~\eqref{eq:GreenFunc1} and~\eqref{eq:GreenFunc2}. This is the most important part in this paper. In previous works (e.g.~\cite{Deng16,DW18,Koike21}), a differential equation technique is used to obtain pointwise estimates; however, this alone is not sufficient when there are several point particles. Nonetheless, the differential equation technique is also important, so let us explain this first.

In~\cite{Koike21}, two Green's functions $G_T$ and $G_R$ appeared, which are defined as follows:
\begin{equation}
  G_T(x,t)\coloneqq \mathcal{L}^{-1}\left[ \frac{2}{\lambda+2}\mathcal{L}[G] \right](x,t), \quad G_R(x,t)\coloneqq (G-G_T)(x,t)
  \begin{pmatrix}
    1 & 0 \\
    0 & -1
  \end{pmatrix},
\end{equation}
where $\lambda=s/\sqrt{\nu s+c^2}$. According to~\eqref{def:Gij}, these are simply
\begin{equation}
  G_T=G_{1,0}, \quad G_R=G_{0,1}.
\end{equation}
Now the differential equation technique uses the relation $\partial_x \mathcal{L}[G](x,s)=-\lambda \mathcal{L}[G](x,s)$ for $x>0$, which is a consequence of~\eqref{prop:IntEq:proof:eq:LaplaceTransformedG}. Because of this relation, $G_T$ satisfies a simple first order ODE:
\begin{equation}
  \label{eq:GT_diff_eq}
  \partial_x G_T(x,t)=2G_T(x,t)-2G(x,t) \quad (x>0).
\end{equation}
Solving this, we obtain
\begin{equation}
  \label{eq:GT_formula}
  G_T(x,t)=2\int_{-\infty}^{0}e^{2z}G(x-z,t)\, dz \quad (x>0).
\end{equation}
The fundamental solution $G$ in the integrand have good pointwise estimates due to the work by Zeng~\cite[Theorem~5.8]{LZ97}.

\begin{prop}
  \label{prop:PWE_G}
  For any integer $k\geq 0$, there exists a positive constant $C=C_k>0$ such that
  \begin{equation}
    \left| \partial_{x}^{k}G(x,t)-\partial_{x}^{k}G^*(x,t)-e^{-\frac{c^2}{\nu}t}\sum_{j=0}^{k}\delta^{(k-j)}(x)Q_j(t) \right| \leq C(t+1)^{-1/2}t^{-(k+1)/2}\left( e^{-\frac{(x-ct)^2}{Ct}}+e^{-\frac{(x+ct)^2}{Ct}} \right),
  \end{equation}
  where
  \begin{equation}
    G^*(x,t)=\frac{1}{2(2\pi \nu t)^{1/2}}e^{-\frac{(x-ct)^2}{2\nu t}}
    \begin{pmatrix}
      1  & -\frac{1}{c} \\
      -c & 1
    \end{pmatrix}
    +\frac{1}{2(2\pi \nu t)^{1/2}}e^{-\frac{(x+ct)^2}{2\nu t}}
    \begin{pmatrix}
      1 & \frac{1}{c} \\
      c & 1
    \end{pmatrix},
  \end{equation}
  $\delta^{(k)}(x)$ is the $k$-th derivative of the Dirac delta function, and $Q_j=Q_j(t)$ is a $2\times 2$ polynomial matrix.
\end{prop}

Combining~\eqref{eq:GT_formula} and Proposition~\ref{prop:PWE_G}, we can show the following~\cite[Appendix~A]{Koike21}:\footnote{We only considered the case of $x>0$ above but the case of $x<0$ is similar.}
\begin{equation}
  \label{eq:PWE_GT}
  |\partial_{x}^{k}G_T(x,t)|\leq C(t+1)^{-1/2}t^{-k/2}\left( e^{-\frac{(x-ct)^2}{Ct}}+e^{-\frac{(x+ct)^2}{Ct}} \right)+Ce^{-\frac{|x|+t}{C}}
\end{equation}
for $(x,t)\in \mathbb{R}\backslash \{ 0 \}\times (0,\infty)$ and for any integer $k\geq 0$. On the other hand, by~\eqref{eq:GT_diff_eq}, we have the relation
\begin{equation}
  \label{eq:GR_relation}
  G_R(x,t)=(G-G_T)(x,t)
  \begin{pmatrix}
    1 & 0 \\
    0 & -1
  \end{pmatrix}
  =-\frac{1}{2}\partial_x G_T(x,t)
  \begin{pmatrix}
    1 & 0 \\
    0 & -1
  \end{pmatrix}
\end{equation}
for $x>0$ (we have a similar formula for $x<0$). From~\eqref{eq:PWE_GT} and~\eqref{eq:GR_relation}, we were able to obtain necessary pointwise estimates for $G_T$ and $G_R$. However, for Green's functions appearing in~\eqref{eq:GreenFunc1} and~\eqref{eq:GreenFunc2}, they do not seem to satisfy such simple ODEs. This is the main difficulty that needs to be resolved.

Our approach to this problem is as follows. First, we consider $G_{i,j}$ defined by~\eqref{def:Gij} instead of their infinite sum such as $G_{0+}$. For $G_{i,j}$, we can find a simple ODE, for example,
\begin{equation}
  \partial_x G_{2,0}(x,t)=2G_{2,0}(x,t)-2G_{1,0}(x,t) \quad (x>0).
\end{equation}
Since we already have pointwise estimates for $G_T=G_{1,0}$, we can obtain those for $G_{2,0}$ by solving this ODE. This then gives pointwise estimates for any finite partial sum of the infinite series in~\eqref{eq:GreenFunc1} and~\eqref{eq:GreenFunc2}. Finally, to analyze the remainders, we use a simple fact: products in the Laplace transformed side are convolutions in time. And to analyze these convolutions, we look into the asymptotic and the analyticity structure in the Laplace variable and apply integration by parts in time as many time as needed to obtain necessary gain of decay (see the proof of Proposition~\ref{prop:GreenPWE}).

Now to state concisely the main result of this section (Proposition~\ref{prop:GreenPWE}), we first give some definitions.

\begin{dfn}
  \label{dfn:types}
  Let $X_H \subset \mathbb{R}$ be an open set.
  \begin{enumerate}[label=(\roman*)]
    \item A function $H\colon X_H \times (0,\infty)\to \mathbb{R}$ is said to be of Type $l$ ($l\in \mathbb{Z}_{\geq 0}$) on $X_H$ if for any integer $k\geq 0$, there exists a positive constant $C=C_k>0$ such that
      \begin{equation}
        |\partial_{x}^{k}H(x,t)|\leq C(t+1)^{-1/2}t^{-(k+l)/2}\left( e^{-\frac{(x-ct)^2}{Ct}}+e^{-\frac{(x+ct)^2}{Ct}} \right)+Ce^{-\frac{|x|+t}{C}} \quad (x\in X_H,t>0).
      \end{equation}
    \item A function $H\colon X_H\times (0,\infty)\to \mathbb{R}$ is said to be of Type $R$ on $X_H$ if $H$ is of Type $1$ on $X_H$ and there exists a (possibly zero) matrix $C_H \in \mathbb{R}^{2\times 2}$ such that $H-(\partial_x G^*)C_H$ is of Type $2$ on $X_H$.
    \item A function\footnote{More precisely, it might contain a Dirac delta singularity but $H-G$ should be a usual function.} $H\colon X_H \times (0,\infty)\to \mathbb{R}$ is said to be of Type $T$ on $X_H$ if $H-G$ is of Type $R$ on $X_H$.
  \end{enumerate}
\end{dfn}

The main result of this section is the following.

\begin{prop}
  \label{prop:GreenPWE}
  The functions defined in~\eqref{eq:GreenFunc1} and~\eqref{eq:GreenFunc2} are of the type listed in Table~\ref{table:prop:GreenPWE}. The table reads as follows: the function $H$ is of Type TYPE on the set $X_H$.
  \begin{table}[h]
    \caption{Types of Green's functions}
    \centering
    \begin{tabular}{l | c c || l | c c}
      \hline
      \addstackgap{$H$} & \addstackgap{TYPE} & \addstackgap{$X_H$} &
      \addstackgap{$H$} & \addstackgap{TYPE} & \addstackgap{$X_H$} \\
      \hline
      \addstackgap{$G_{++}$} & \addstackgap{T} & \addstackgap{$\mathbb{R}$} &
      \addstackgap{$G^{++}$} & \addstackgap{R} & \addstackgap{$(2,\infty)$} \\ [0.5ex]

      \addstackgap{$G_{0+}$} & \addstackgap{T} & \addstackgap{$(0,\infty)$} &
      \addstackgap{$G^{0+}$} & \addstackgap{R} & \addstackgap{$(1,\infty)$} \\ [0.5ex]

      \addstackgap{$G_{-+}$} & \addstackgap{T} & \addstackgap{$(1,\infty)$} &
      \addstackgap{} & \addstackgap{} & \addstackgap{} \\ [0.5ex]

      \addstackgap{$G_{0b}^{+}$} & \addstackgap{T} & \addstackgap{$(1,\infty)$} &
      \addstackgap{$G_{1b}^{+}$} & \addstackgap{T} & \addstackgap{$(1,\infty)$} \\ [0.5ex]
      \addstackgap{$G_{+0}$} & \addstackgap{T} & \addstackgap{$(-\infty,0)$} &
      \addstackgap{$G^{+0}$} & \addstackgap{R} & \addstackgap{$(1,\infty)$} \\ [0.5ex]

      \addstackgap{$G_{00}$} & \addstackgap{T} & \addstackgap{$(-1,1)$} &
      \addstackgap{$G^{00}$} & \addstackgap{R} & \addstackgap{$(0,2)$} \\ [0.5ex]

      \addstackgap{$G_{-0}$} & \addstackgap{T} & \addstackgap{$(0,\infty)$} &
      \addstackgap{$G^{-0}$} & \addstackgap{R} & \addstackgap{$(-\infty,1)$} \\ [0.5ex]

      \addstackgap{$G_{0b}^{0}$} & \addstackgap{T} & \addstackgap{$(0,1)$} &
      \addstackgap{$G_{1b}^{0}$} & \addstackgap{T} & \addstackgap{$(0,1)$} \\ [1.0ex]
      \hline
    \end{tabular}
    \label{table:prop:GreenPWE}
  \end{table}
\end{prop}

\begin{rem}
  Green's functions consisting only of $G_{i,j}$ with $j\geq 1$ are of Type $R$ and others are of Type~$T$. According to Remark~\ref{rem:IntEqInterpretation}, this means that Green's functions consisting only of $G_{i,j}$ resulting from one or more reflections are of Type $R$.
\end{rem}

To prove this proposition, we first prepare two lemmas.

\begin{lem}
  \label{lem:Gij}
  We have
  \begin{equation}
    |\partial_{x}^{k}G_{i,j}(x,t)|\leq C(t+1)^{-(i+j)/2}t^{-(k+1-i)/2}\left( e^{-\frac{(x-ct)^2}{Ct}}+e^{-\frac{(x+ct)^2}{Ct}} \right)+Ce^{-\frac{|x|+t}{C}}
  \end{equation}
  for $(x,t)\in \mathbb{R}\backslash \{ 0 \}\times (0,\infty)$ and for any integer $k\geq 0$.
\end{lem}

\begin{proof}
  Assume that $x>0$. The case of $x<0$ is similar. Note first that by~\eqref{def:Gij} and~\eqref{prop:IntEq:proof:eq:LaplaceTransformedG}, we have
  \begin{equation}
    G_{i,j}(x,t)=2^{-j}\partial_{x}^{j}G_{i+j,0}(x,t)
    \begin{pmatrix}
      -1 & 0 \\
      0 & 1
    \end{pmatrix}^j.
  \end{equation}
  So it suffices to consider $G_{i,0}$. Note that by Proposition~\ref{prop:PWE_G}, the lemma is true for $G_{0,0}=G$.

  We next note that $G_{i,0}$ satisfies the following ODE:
  \begin{equation}
    \partial_x G_{i,0}(x,t)=2G_{i,0}(x,t)-2G_{i-1,0} \quad (x>0).
  \end{equation}
  Solving this, we obtain
  \begin{equation}
    G_{i,0}(x,t)=2\int_{-\infty}^{0}e^{2z}G_{i-1,0}(x-z,t)\, dz \quad (x>0).
  \end{equation}
  If we now assume that the lemma is true for $G_{i-1,0}$, then~\cite[Lemma~A.1]{Koike21} and a simple estimate
  \begin{equation}
    \int_{-\infty}^{0}e^{2z}e^{-\frac{|x-z|+t}{C}}\, dz\leq e^{-\frac{|x|+t}{C}}\int_{-\infty}^{0}e^{z}\, dz=e^{-\frac{|x|+t}{C}}
  \end{equation}
  imply that the lemma is also true for $G_{i,0}$. This proves the lemma.
\end{proof}

\begin{lem}
  \label{lem:LaplaceAsymptotics}
  Let $\lambda=s/\sqrt{\nu s+c^2}$ and $r=\lambda^2/(\lambda+2)^2$. Then the following properties are satisfied:
  \begin{enumerate}[label=(\Alph*)]
    \item There exist $\sigma_0>0$ and $r_0 \in (0,1)$ such that $\Re \lambda \geq -r_0$ and $|re^{-2\lambda}|\leq r_0$ for all $s\in \mathbb{C}\backslash (-\infty,-c^2/\nu]$ with $\Re s\geq -\sigma_0$.
    \item We have
      \begin{equation}
        \Re \lambda \geq \frac{\sqrt{|s|}}{2\sqrt{\nu}}+O\left( \frac{1}{\sqrt{|s|}} \right) \quad (|s|\to \infty; \Re s>-c^2/\nu).
      \end{equation}
  \end{enumerate}
\end{lem}

\begin{proof}
  Note first that $\Re \lambda\geq 0$ for $\Re s\geq 0$. From this, it follows that for any $M>0$, there exists $r_0 \in (0,1)$ such that $|re^{-2\lambda}|\leq r_0$ for all $s$ with $\Re s\geq 0$ and $|s|\leq M$. Then, if (B) is proved, (A) follows easily.

  Now we prove (B). Let $s$ be a sufficiently large complex number with $\Re s>-c^2/\nu$. We can write it as $s=|s|e^{i\theta}$ with $\theta \in [-2\pi/3,2\pi/3]$. Then, from
  \begin{equation}
    \lambda=\frac{s}{\sqrt{\nu s}}\frac{\sqrt{\nu s}}{\sqrt{\nu s+c^2}}=\frac{\sqrt{s}}{\sqrt{\nu}}\left[ 1+O\left( \frac{1}{s} \right) \right]=\frac{\sqrt{|s|}}{\sqrt{\nu}}e^{\frac{i\theta}{2}}\left[ 1+O\left( \frac{1}{s} \right) \right],
  \end{equation}
  we obtain
  \begin{equation}
    \Re \lambda=\frac{\sqrt{|s|}}{\sqrt{\nu}}\cos \left( \frac{\theta}{2} \right)+O\left( \frac{1}{\sqrt{|s|}} \right) \geq \frac{\sqrt{|s|}}{2\sqrt{\nu}}+O\left( \frac{1}{\sqrt{|s|}} \right).
  \end{equation}
  This ends the proof of the lemma.
\end{proof}

\begin{proof}[Proof of Proposition~\ref{prop:GreenPWE}]
  We shall only prove that $G_{0+}$ is of Type $T$ on $X_H=(0,\infty)$. Others can be treated in a similar manner. Now fix $x\in X_H$. By~\eqref{def:Gij},~\eqref{eq:GreenFunc1},~\eqref{def:r}, and~\eqref{prop:IntEq:proof:eq:LaplaceTransformedG}, we have
  \begin{align}
    G_{0+}(x,t)
    & =G_{1,0}(x,t)+\sum_{i=1}^{\infty}G_{1,2i}(x+2i,t) \\
    & =G_{1,0}(x,t)+\sum_{i=1}^{\infty}\mathcal{L}^{-1}\left[ \mathcal{L}[G]\frac{2}{\lambda+2}r^i e^{-2i\lambda} \right](x,t) \\
    & =G_{1,0}(x,t)+\mathcal{L}^{-1}\left[ \mathcal{L}[G]\frac{2re^{-2\lambda}}{(\lambda+2)(1-re^{-2\lambda})} \right](x,t) \\
    & \eqqcolon G_{1,0}(x,t)+G_{1,2\infty}(x,t).
  \end{align}
  Next, note that
  \begin{equation}
    (G_{1,0}-G)(x,t)=-\mathcal{L}^{-1}\left[ \mathcal{L}[G]\frac{\lambda}{\lambda+2} \right](x,t)
    =-G_{0,1}(x,t)
    \begin{pmatrix}
      1 & 0 \\
      0 & -1
    \end{pmatrix}.
  \end{equation}
  Therefore, in order to show that $G_{0+}$ is of Type $T$ on $X_H$, it suffices to show that $G_{0,1}$ is of Type $R$ and $G_{1,2\infty}$ is of Type $2$ on the same set. That $G_{0,1}$ is of Type $R$ on $X_H$ is proved by noting that
  \begin{equation}
    G_{0,1}(x,t)-\frac{1}{2}\partial_x G(x,t)
    \begin{pmatrix}
      -1 & 0 \\
      0 & 1
    \end{pmatrix}
    =-\mathcal{L}^{-1}\left[ \mathcal{L}[G]\frac{\lambda^2}{2(\lambda+2)} \right](x,t)
    \begin{pmatrix}
      1 & 0 \\
      0 & -1
    \end{pmatrix}
    =-\frac{1}{4}\partial_{x}^{2}G_{1,0}(x,t)
    \begin{pmatrix}
      1 & 0 \\
      0 & -1
    \end{pmatrix}
  \end{equation}
  and applying Proposition~\ref{prop:PWE_G} and Lemma~\ref{lem:Gij}.

  Let us next show that $G_{1,2\infty}$ is of Type $2$ on $X_H$. Let
  \begin{equation}
    \omega(t)=\mathcal{L}^{-1}\left[ \frac{2re^{-\lambda}}{(\lambda+2)(1-re^{-2\lambda})} \right](t).
  \end{equation}
  Then, since products in the Laplace transformed side are convolutions in time, we have
  \begin{equation}
    G_{1,2\infty}(x,t)=\int_{0}^{t}\omega(t-s)G(x,s)\, ds.
  \end{equation}
  By Lemma~\ref{lem:LaplaceAsymptotics} (A), there exists $\sigma_0>0$ such that $\mathcal{L}[\omega](s)$ is analytic in the half-space $\{ \Re s>-\sigma_0 \}$. Moreover, by Lemma~\ref{lem:LaplaceAsymptotics} (B), $\mathcal{L}[\omega](s)$ decays exponentially fast as $|s|\to \infty$ on any vertical line with $\Re s>-c^2/\nu$. Therefore, there exist $\alpha,C>0$ such that
  \begin{equation}
    \label{prop:GreenPWE:eq:omegaExpDecay}
    |\omega(t)|\leq Ce^{-\alpha t} \quad (t\geq 0).
  \end{equation}
  Next, let
  \begin{equation}
    \omega_1(t)=\int_{0}^{t}\omega(s)\, ds.
  \end{equation}
  Then we have $\mathcal{L}[\omega_1](s)=\mathcal{L}[\omega](s)/s$. Since $r=O(s^2)$ as $|s|\to 0$, the origin $s=0$ is a removable singularity of $\mathcal{L}[\omega_1](s)$, which implies that there exist $\alpha,C>0$ such that
  \begin{equation}
    \label{prop:GreenPWE:eq:omega1ExpDecay}
    |\omega_1(t)|\leq Ce^{-\alpha t} \quad (t\geq 0).
  \end{equation}
  Similarly, if we define
  \begin{equation}
    \omega_2(t)=\int_{0}^{t}\omega_1(s)\, ds,
  \end{equation}
  we have
  \begin{equation}
    \label{prop:GreenPWE:eq:omega2ExpDecay}
    |\omega_2(t)|\leq Ce^{-\alpha t} \quad (t\geq 0).
  \end{equation}
  Now we divide $G_{1,2\infty}$ as follows:
  \begin{equation}
    G_{1,2\infty}(x,t)=I(x,t)+J(x,t)\coloneqq \int_{0}^{t/2}\omega(t-s)G(x,s)\, ds+\int_{t/2}^{t}\omega(t-s)G(x,s)\, ds.
  \end{equation}

  Let us first show that $I$ is of Type $2$ on $X_H$. By Proposition~\ref{prop:PWE_G} and~\eqref{prop:GreenPWE:eq:omegaExpDecay}, we have (note that $x>0$ here)
  \begin{align}
    |I(x,t)|
    & \leq C\int_{0}^{t/2}e^{-\alpha(t-s)}s^{-1/2}e^{-\frac{(x-cs)^2}{Cs}}\, ds \\
    & =C\int_{0}^{t/2}e^{-\alpha(t-s)}s^{-1/2}e^{-\frac{(x-cs)^2}{Cs}}\chi_{\{ c(t-s)\leq |x-ct|/4 \}}\, ds+C\int_{0}^{t/2}e^{-\alpha(t-s)}s^{-1/2}e^{-\frac{(x-cs)^2}{Cs}}\chi_{\{ c(t-s)>|x-ct|/4 \}}\, ds \\
    & \leq Ce^{-\frac{\alpha}{2}t}e^{-\frac{(x-ct)^2}{Ct}}\int_{0}^{t/2}s^{-1/2}\, ds+Ce^{-\frac{\alpha}{4}t}e^{-\frac{\alpha}{8c}|x-ct|}\int_{0}^{t/2}s^{-1/2}\, ds \\
    & \leq Ce^{-\frac{t}{C}}e^{-\frac{(x-ct)^2}{Ct}}+Ce^{-\frac{|x|+t}{C}}.
  \end{align}
  Here, $\chi_{\{ S\}}$ is the indicator function of the set $S$. The derivatives $\partial_{x}^{k}I(x,t)$ can be treated similarly by modifying the definition of $\omega(t)$.\footnote{Directly applying Proposition~\ref{prop:PWE_G} for the derivatives of $G$ results in apparently diverging integrals. We can circumvent this problem by noting that applying $\partial_{x}^{k}$ is equivalent to multiplying $(-\lambda)^k$ in the Laplace transformed side, and the divergence of $\lambda^k$ as $|s|\to \infty$ is then absorbed by the exponential factor in $\mathcal{L}[\omega](s)$.}~This shows that $I$ is of Type $2$ on $X_H$.

  We next show that $J$ is of Type $2$ on $X_H$. First, assume that $t\leq 1$. Then by Proposition~\ref{prop:PWE_G} and~\eqref{prop:GreenPWE:eq:omegaExpDecay}, we have
  \begin{equation}
    |J(x,t)|\leq C\int_{t/2}^{t}e^{-\alpha(t-s)}s^{-1/2}e^{-\frac{(x-cs)^2}{Cs}}\, ds\leq Ce^{-\frac{(x-ct)^2}{Ct}}.
  \end{equation}
  So we may now assume that $t>1$. By applying integration by parts twice, we obtain
  \begin{align}
    \label{prop:GreenPWE:eq:JIntByParts}
    \begin{aligned}
      J(x,t)
      & =\int_{t/2}^{t}\partial_t \omega_1(t-s)G(x,s)\, ds \\
      & =\omega_1(t/2)G(x,t/2)+\int_{t/2}^{t}\omega_1(t-s)\partial_t G(x,s)\, ds \\
      & =\omega_1(t/2)G(x,t/2)+\omega_2(t/2)\partial_t G(x,t/2)+\int_{t/2}^{t}\omega_2(t-s)\partial_{t}^{2}G(x,s)\, ds.
    \end{aligned}
  \end{align}
  To analyze the terms on the righ-hand side, we first show that
  \begin{equation}
    \label{prop:GreenPWE:eq:KeyIneq:Eq1}
    e^{-\frac{\alpha}{2}(t-s)}e^{-\frac{(x-cs)^2}{Cs}}\leq e^{-\frac{(x-ct)^2}{Ct}}
  \end{equation}
  for $0<s\leq t$ if $C\geq 2c^2/\alpha$: this inequality is equivalent to
  \begin{equation}
    \frac{x^2-2ctx+c^2 t^2}{Ct}\leq \frac{x^2-2csx+c^2 s^2}{Cs}+\frac{\alpha}{2}(t-s),
  \end{equation}
  and the right-hand side minus the left-hand side equals to
  \begin{equation}
    \left( \frac{1}{Cs}-\frac{1}{Ct} \right)x^2+\left( \frac{\alpha}{2}-\frac{c^2}{C} \right)(t-s)\geq 0.
  \end{equation}
  Now using~\eqref{prop:GreenPWE:eq:KeyIneq:Eq1} with $s=t/2$, \eqref{prop:GreenPWE:eq:omega1ExpDecay}, and~\eqref{prop:GreenPWE:eq:omega2ExpDecay}, we obtain
  \begin{equation}
    \label{prop:GreenPWE:eq:J_BoundaryTerm}
    |\omega_1(t/2)G(t/2)|+|\omega_2(t/2)\partial_t G(x,t/2)|\leq Ce^{-\frac{\alpha}{4}t}e^{-\frac{(x-ct)^2}{Ct}}.
  \end{equation}
  Next, by~\eqref{eq:FundamentalSolution}, Proposition~\ref{prop:PWE_G}, \eqref{prop:GreenPWE:eq:omega2ExpDecay}, and~\eqref{prop:GreenPWE:eq:KeyIneq:Eq1}, we obtain
  \begin{equation}
    \label{prop:GreenPWE:eq:KeyIneq}
    \int_{t/2}^{t}|\omega_2(t-s)\partial_{t}^{2}G(x,s)|\, ds\leq C(t+1)^{-3/2}\int_{t/2}^{t}e^{-\alpha(t-s)}e^{-\frac{(x-cs)^2}{Cs}}\, ds\leq C(t+1)^{-3/2}e^{-\frac{(x-ct)^2}{Ct}}.
  \end{equation}
  Then,~\eqref{prop:GreenPWE:eq:JIntByParts},~\eqref{prop:GreenPWE:eq:J_BoundaryTerm},~and~\eqref{prop:GreenPWE:eq:KeyIneq} imply
  \begin{equation}
    |J(x,t)|\leq C(t+1)^{-3/2}e^{-\frac{(x-ct)^2}{Ct}}.
  \end{equation}
  The derivatives $\partial_{x}^{k}J(x,t)$ can be treated similarly by modifying the definition of $\omega(t)$ (use integration by parts $k+2$ times). This shows that $J$ is of Type $2$ on $X_H$, which ends the proof of the lemma.
\end{proof}

\subsection{Concluding remarks}
The final step of the proof is the nonlinear estimates. Although this step is very important, it is almost identical to those in the proofs of~\cite[Theorem~1.2]{Koike21} and~\cite[Theorem~2.1]{Koike20-p}. In these previous works, that Green's functions $G_T$ and $G_R$ (see the beginning of Section~\ref{sec:PWE_GreensFunctions}) are of Type $T$ and Type $R$, respectively, is crucially used; the corresponding properties are established by Proposition~\ref{prop:GreenPWE}, and we can imitate the calculations in the previous works to prove Theorems~\ref{thm:main_H4} and~\ref{thm:main}. As the necessary calculations are lengthy and too much of a repetition, we omit the detail and end the proof here.

\subsubsection*{Acknowledgements}
I thank Shih-Hsien Yu for informing me about his preprint~\cite{LY_preprint} when I visited National University of Singapore in 2019, which was financially supported by Grant-in-Aid for JSPS Research Fellow (Grant Number 18J20574). This motivated me to consider the problem presented in this paper, and Proposition~\ref{prop:IntEq} emerged as an application of one of the ideas in their paper. This work was financially supported by Grant-in-Aid for JSPS Research Fellow (Grant Number 20J00882).

\bibliographystyle{amsplain}
\bibliography{kai-2021-1}

\end{document}